\documentclass[10 pt,english]{article}
\usepackage[activeacute]{babel}
\usepackage{amsfonts} 
\usepackage{mathrsfs} 
\usepackage{amssymb, amsmath, amsfonts}
\usepackage{makeidx}
\usepackage{epsfig}
\usepackage{color,url}
\usepackage[T1]{fontenc}
\usepackage{graphics,graphicx,graphpap}
\usepackage{amsthm}
\usepackage{pstricks-add}
\usepackage[ shadow, colorinlistoftodos,textsize=tiny, obeyFinal]{todonotes}
\setuptodonotes{fancyline, backgroundcolor=gray!10,bordercolor=gray}

\newcommand\blfootnote[1]{%
  \begingroup
  \renewcommand\thefootnote{}\footnote{#1}%
  \addtocounter{footnote}{-1}%
  \endgroup
}

\if@twoside \oddsidemargin 10pt \evensidemargin 10pt \marginparwidth 20pt
 \else \oddsidemargin 10pt \evensidemargin 10pt \marginparwidth 20pt \fi

\newcommand{\ZZ}{\mathbb{Z}}

\newcommand{\Aut}{\mathrm{Aut}}
\newcommand{\G}{\Gamma}

\newcommand{\cC}{\mathcal{C}}

\newcommand{\m}{\mathcal{M}}
\newcommand{\cM}{\mathcal{M}}
\newcommand{\cS}{\mathcal{S}}

\newcommand{\Stab}{\mathrm{Stab}}

\newtheorem{theorem}{Theorem}[section]
\newtheorem{proposition}[theorem]{Proposition}

\newtheorem{lemma}[theorem]{Lemma}

\newtheorem{question}[theorem]{Question}

\theoremstyle{definition}

\begin{document}

\begin{center}
{\Large{\textbf{An infinite family of simple graphs underlying chiral, orientable reflexible and non-orientable rotary maps}}} \\ [+4ex]
{\large Isabel Hubard{\small$^{a}$}, Primo\v{z} Poto\v{c}nik{\small$^{b,c,*}$} and Primo\v z \v Sparl{\small$^{c, d, e}$}}
\\ [+2ex]
{\it \small 
$^a$Institute of Mathematics, National Autonomous University of Mexico, 04510 Mexico City, Mexico.\\
$^b$University of Ljubljana, Faculty of Mathematics and Physics, Jadranska 19, SI-1000 Ljubljana, Slovenia.\\
$^c$Institute of Mathematics, Physics and Mechanics, Jadranska 19, SI-1000 Ljubljana, Slovenia.\\ 
$^d$University of Ljubljana, Faculty of Education, Kardeljeva plo\v s\v cad 16, SI-1000 Ljubljana, Slovenia.\\
$^e$University of Primorska, Institute Andrej Maru\v{s}i\v{c}, Muzejski trg 2, SI-6000 Koper, Slovenia.
}
\end{center}

\blfootnote{
Email addresses: 
isahubard@im.unam.mx (Isabel Hubard),
primoz.potocnik@fmf.uni-lj.si (Primo\v z Poto\v cnik),
primoz.sparl@pef.uni-lj.si (Primo\v z \v Sparl)
\\
* - corresponding author
}

\begin{abstract}
In this paper, we provide the first known infinite family of simple graphs, each of which is the skeleton of a chiral map, a skeleton of a reflexible map on an orientable surfaces, as well as
a skeleton of a reflexible map on a non-orientable surface. This family consists of all 
lexicographic product $C_n[mK_1]$, where $m\ge 3$, $n = sm$, with $s$ an integer not divisible by $4$.
This answers a question posed in [S.\ Wilson, Families of regular graphs in regular maps,
{\em Journal of Combinatorial Theory, Series B} 85 (2002), 269--289].
\end{abstract}


\section{Introduction}

A {\em map} is a connected finite graph 
(called the {\em skeleton of the map}),
 viewed as a 1-dimensional CW complex, embedded on a closed surface  
 such that 
 when it is removed from the surface,  the connected components
(called {\em faces}) are all homeomorphic to an open disk.
The edges and vertices of the skeleton are referred to
also as edges and vertices of the map.
The map is called {\em polytopal} provided that the boundary of each face is a cycle in the graph and that every edge lies on the boundary of two distinct faces.
The term ``polytopal'' here comes from the fact that such a map can be considered as an abstract polytope of rank $3$; see for example \cite{ARP}. 
One should also note that polytopal maps are examples of closed 2-cellular embeddings of graphs, in both of the two standard meanings of
this notion (see, for example, the discussion in \cite[Section 1]{EllZha}).
All maps  appearing in this paper will be polytopal and all graphs will be simple.

An automorphism of the skeleton of a map that extends to a homeomorhism of the underlying surface is  called an {\em automorphism of the map} and the set of all such automorphisms forms a group, called the {\em automorphism group of the map}. 
If the automorphism group is rich enough so that the stabiliser of each vertex contains a cyclic group acting transitively on the neighbouring edges (and faces)
and, in addition, so that the stabiliser of each face contains a cyclic group acting transitively on the vertices (and edges) of the face, then the map is called {\em rotary}. 
Note that the automorphism group of a rotary map acts transitively on the vertices, on the faces, on the oriented edges (arcs), as well as on incident vertex-face pairs of the map. 
In particular, all faces in a rotary map have equal co-valence (where by the {\em co-valence of a face} we mean 
 the number of edges lying on the boundary of the face). If a rotary map has co-valence $p$ and the skeleton has valence $q$, then we say that the map
 has {\em type} $\{p,q\}$.

If, in addition, the face-stabiliser in a rotary map also contains
an automorphism that acts as a reflection on the face (which is equivalent to requiring that the stabiliser of an arc in the face stabiliser is non-trivial), then the map is called {\em reflexible}. 
Note that every rotary map on a non-orientable surface is reflexible. 
On the other hand, if the underlying surface is orientable, then rotary maps that are not reflexible exist and are called {\em chiral}. 
Reflexible maps can thus be {\em orientable} (if the underlying surface is orientable) or {\em non-orientable} (if the surface is non-orientable). 

There is a vast literature on different aspects of rotary maps, but typically, the topic is approached from one of the following three points of view:
\begin{itemize}
\setlength{\itemsep}{0pt}
\item Given a fixed (orientable or non-orientable) surface $\mathcal{S}$, classify all rotary maps on $\mathcal{S}$;
\item Given a finite group $G$ (or a family thereof), find all rotary maps whose automorphism group is isomorphic to $G$;
\item Given a finite connected graph $\Gamma$ (or a family thereof), determine all rotary maps whose skeleton is $\Gamma$.
\end{itemize}

This paper falls into the third of the above categories. In particular, we shall address the following question, which was originally posed by Steve Wilson in \cite{wilson2002families}.
\begin{question}
Does there exist a graph which is the skeleton of representatives from all three classes of rotary maps: 
non-orientable reflexible, orientable reflexible, and chiral?
\end{question}

As was observed already in \cite{wilson2002families}, the answer to this question is affirmative since the complete multipartite graph $K_{3,3,3}$ is an example of such a graph. However, the existence of
an infinite family of graphs, each of which underlies maps from all three classes, has not been known until now.
The main purpose of this paper is to exhibit such an infinite family. In particular, we shall prove the following theorem:

\begin{theorem}
\label{the:main}
Let $s$ be an arbitrary positive integer not divisible by $4$, let $m\geq 3$ be an odd integer, let $n = sm$, 
and let $\Gamma$ be the lexicographic product $C_n[mK_1]$.
Then $\Gamma$ is the skeleton of
\begin{itemize}
\item a chiral polytopal map of type $\{mn, 2m\}$ and genus $1+m(n\frac{m-1}{2} - 1)$, and
\item an orientable reflexible polytopal map of type $\{n, 2m\}$ and genus $1+m(n\frac{m-1}{2} - m)$, and
\item a non-orientable reflexible polytopal map  of type $\{2n, 2m\}$ and genus $2+m(n(m-1)-m)$.
\end{itemize}
\end{theorem}

\section{Rotary maps arising from graphs}
\label{sec:maps}

In this section, we review some definitions pertaining to rotary maps and their automorphisms.
\color{black}

Let $\G$ be the skeleton of a polytopal map $\m$ and assume that $\G$ is a simple graph. 
A cycle  $C$ that constitutes the boundary of a face $f$ is called a {\em facial cycle} of $\m$. 
Since the face of a polytopal map of type $\{p,q\}$, with $q\geq 3$, is uniquely determined by its facial cycles, we shall often think of faces both as parts of the underlying surface as well as the cycles in the skeleton.
By definition of polytopality of maps, the set $\cC$ of all facial cycles has the property that every edge of $\Gamma$ belongs to precisely two cycles in $\cC$. 
Such a set of cycles in a graph is called a {\em cycle double cover} (see \cite{Jamishy}, for example). 
In short, the set of facial cycles in a polytopal map is a cycle double cover of the skeleton.

Let $\cC$ be a cycle double cover of a simple graph $\G$. 
For a vertex $v\in\G$, its {\em vertex figure with respect to $\cC$} is defined as the graph whose vertices are the edges of $\G$ incident to $v$, and
two such edges are adjacent in the vertex figure whenever they are consecutive edges of a cycle in $\cC$.
Note that the valence of a vertex in a vertex figure is at most $2$ (since every edge of $\G$ belongs to exactly two cycles in $\cC$).
Moreover,  if $\cC$ is the set of facial cycles of a polytopal map, then the vertex figure at a vertex $v$ of valence $q$ is in fact isomorphic to the cycle $C_q$ (if $q\ge 3$)
or to $K_2$ (if $q=2$).

%
%
The following lemma, providing the converse of the above observations,
 is a straightforward folklore result and has appeared in similar forms in several publications (see, for example, \cite[Lemma 3.9]{PotVid}). Here
we provide just a brief sketch of the proof.

\begin{lemma}
\label{lemma:polyhedron}
  Let $\G$ be a connected simple graph of minimal valence at least $3$ and let $\cC$ be a cycle double cover of $\G$ such that
 the vertex figure of each vertex of $\G$ with respect to $\cC$ is connected (and thus isomorphic to a cycle).
 Let $\cS$ be the topological space obtained from $\G$ (viewed as a 1-dimensional CW complex) by 
  gluing a copy of a closed disk to each cycle of $\cC$ homeomorphically along the boundary of the disk. Then $\cS$ is a closed surface and
  the embedding of $\G$ onto $\cS$ is a polytopal map $\m(\G,\cC)$ whose set of facial cycles is $\cC$. 
  Moreover, the automorphism group $\Aut(\m(\G,\cC))$ consists of all automorphisms of $\Aut(\G)$ that induce a permutation of the cycles of $\cC$. 
\end{lemma}

The proof of the above lemma can be sketched as follows. First, observe that every point of $\cS$ which is an internal point of one of the
closed disks glued to $\G$ has a regular neighbourhood (that is, homeomorphic to an open disk). Similarly, a regular neighbourhood of an internal point of an edge exists due to the fact that every edge of $\G$ belongs to exactly two cycles in $\cC$ (and is thus identified with a point in the boundary of exactly two disks that were
glued to $\G$). 
Finally, since the minimal valence of $\G$ is at least $3$ and since every edge belongs to two cycles in $\cC$, the vertex figure at each vertex (being connected) is a cycle. 
This then enables one to find a regular neighbourhood of each point of $\cS$ that corresponds to a vertex of $\G$. 
In short, $\cS$ is a closed surface with
$\G$ embedded into it. 
Since the boundary of each face is a cycle of $\G$ and since every edge lies on the boundaries of two distinct faces, the corresponding map
is polytopal. 
Finally, a simple fact that every homeomorphism between the boundaries of two closed disks can be extended to a homeomorphism between the disks
 implies that every automorphism of $\G$ that preserves the set of facial cycles $\cC$ extends to a homeomorphism of $\cS$. 
With this we finish the sketch of the proof of Lemma~\ref{lemma:polyhedron}.

A {\em flag} of a polytopal map is an incident vertex-edge-face triple.
Observe that the automorphism group of a polytopal map acts naturally on the set of flags of the map.
In particular, the connectedness implies that the automorphism group is semi-regular (or free) on the flags. 
Thus,  the order of the automorphism group divides the number of flags and is at most twice the number of arcs and at most four times the number of edges. 

Given a polytopal rotary map $\cM$ and a flag  $\Phi=(v,e,f)$ of $\cM$,  there exist automorphisms $\sigma_1$ and $\sigma_2$, acting as a 1-step rotation of the face $f$ and around the vertex $v$, respectively. These automorphisms can be chosen in such a way that $\sigma_1\sigma_2$ is an involution reversing the edge $e$.
The group $\langle \sigma_1, \sigma_2 \rangle \leq \Aut(\m)$ is denoted by $\Aut^+(\m)$, and is often called the {\em rotational group} of $\m$;
the automorphisms $\sigma_1$ and $\sigma_2$ are  then called the {\em distinguished generators of $\Aut^+(\m)$ with respect to the base flag $\Phi$}.
The rotational group of a rotary map $\m$ has index at most $2$ in $\Aut(\m)$; in fact, the index of $\Aut^+(\m)$ in $\Aut(\m)$ is $2$ whenever $\m$ is reflexible and orientable, otherwise $\Aut^+(\m)=\Aut(\m)$.

It is not difficult to see that whenever $\m$ is reflexible there exists an involutory automorphism $\rho$ of the map fixing both $v$ and $f$, and interchanging the two edges of $f$ incident to $v$;
in this case, $\Aut(\m)$ acts transitively on the flags of $\m$.
Moreover, $\langle \rho, \sigma_1 \rangle$ (resp. $\langle \rho, \sigma_2 \rangle$) is a dihedral group: the full automorphism group of the face $f$ (resp.\ of the vertex figure of $v$). Thus, conjugation by $\rho$ sends $\sigma_i$ to $\sigma_i^{-1}$, for $i=1,2$;
we shall say in such a case that $\rho$ {\em inverts} $\sigma_i$.
Recall that if $\m$ is non-orientable reflexible, then $\Aut^+(\m)=\Aut(\m)$ and consequently, the involutory automorphism $\rho$ belongs to $\Aut^+(\m)$.
In contrast, if $\m$ is orientable reflexible, then $\rho \notin \Aut^+(\m)$, and $\Aut(\m)=\langle \sigma_1, \sigma_2, \rho\rangle$.

With the discussion above, the following lemma is straightforward.

\begin{lemma}\label{chiralorreflexible}
    Let $\m$ be a rotary polytopal map of type $\{p,q\}$ with the skeleton $\G$. Let $\sigma_1$ and $\sigma_2$ be the distinguished generators of $G=\Aut^+(\m)$ with respect to some base flag $\Phi$. Let $v$ be the vertex of $\Phi$.
    Then  $(\sigma_1\sigma_2)^2=1$, $\sigma_1$ has order $p$ and maps $v$ to a neighbour $w$ of $v$, while $\sigma_2$ has order $q$, fixes $v$ and acts transitively on its neighbours.
    Moreover,
    \begin{itemize}
        \item $\m$ is non-orientable (and thus reflexible) if and only if there exists an involution $\rho \in G$  fixing $v$ and inverting $\sigma_1$ and $\sigma_2$. 
        In this case $|\Stab_G(vw)|=2$.
        \item $\m$ is reflexible and orientable if and only if there exists an involution $\rho \in\Aut(\m)\setminus G$  fixing $v$ and inverting $\sigma_1$ and $\sigma_2$.  In this case $|\Stab_G(vw)|=1$.
        \item  $\m$ is chiral if and only if there exists no involution $\rho\in\Aut(\m)$ fixing $v$ and inverting $\sigma_1$ and $\sigma_2$. In this case, $|\Stab_G(vw)|=1$.
    \end{itemize}
\end{lemma}


The following result, which can be thought of as a converse of the previous lemma, characterises graphs that can be embedded as skeletons of rotary maps. 
Versions of this result can be found in \cite{gnss} and \cite{surowski}. 
However, additional information that we provide here is somewhat difficult to deduce from the previous results. For this reason, we decided to provide an independent proof.

\begin{theorem}
\label{prop:undergraph}
    Let $q\ge 3$ be an integer, let $\Gamma$ be a $q$-valent connected simple graph, let $v$ be a vertex of $\Gamma$ and let $\sigma_1$ and $\sigma_2$ be automorphisms of $\G$ satisfying the following conditions:
   \begin{enumerate}
     \item[{\rm (i)}] $\sigma_2$  fixes $v$ and the group $\langle \sigma_2 \rangle$ transitively permutes the $q$ neighbours of $v$;
     \item[{\rm (ii)}] $\sigma_1$ maps $v$ to a neighbour $w$ of $v$;
     \item[{\rm (iii)}] $(\sigma_1\sigma_2)^2=1$. 
   \end{enumerate} 
 Let $G = \langle \sigma_1 ,\sigma_2\rangle$ and let $p$ be the order of $\sigma_1$. Then $p \geq 3$, the group $G$ acts transitively on the arcs of $\G$ and the following holds:
    \begin{enumerate}
        \item[{\rm (a)}] If $|\mathrm{Stab}_G(vw)| =1$, then $\G$ is the skeleton of a polytopal rotary map $\m$ of type $\{p,q\}$
        on an orientable surface with $\Aut^+(\m)=G$.
           In this case, $\m$ is reflexible if and only if there exists an involutory automorphism $\rho$ of $\G$ normalising $G$, inverting $\sigma_1$ and fixing $v$.
        \item[{\rm (b)}] If  $|\mathrm{Stab}_G(vw)| =2$ and there exists $\rho\in G$ such that $v\rho=v$ and $\sigma_1^\rho=\sigma_1^{-1}$,
        then $\G$ is the skeleton of a polytopal non-orientable reflexible map $\m$  of type $\{p,q\}$ with $\Aut(\m) = \Aut^+(\m)=G$.
    \end{enumerate}
   In both cases, the set of the facial cycles of $\m$ is the $G$-orbit of the cycle $f$ given by the sequence of vertices $(v, v\sigma_1, v\sigma_1^2, \ldots , v\sigma_1^p)$, and $\sigma_1$, $\sigma_2$ are the distinguished generators of $\Aut^+(\m)$ with respect to the flag $(v, e, f)$, where $e$ is the edge with vertices $v$ and $v\sigma_1^{-1}$.
   Moreover, whenever $\cM$ is reflexible, conjugation by $\rho$ (as given above) also inverts $\sigma_2$.
\end{theorem}

\begin{proof}
For convenience, set $u=v{\sigma_1^{-1}}$.
Let us now deduce some consequences of the conditions given in the theorem. 
First, notice that since $\sigma_1\sigma_2$ is an involution, it follows that $\sigma_1\sigma_2 = \sigma_2^{-1}\sigma_1^{-1}$, and thus 
$$
	w{\sigma_2} = v{\sigma_1\sigma_2} = v{\sigma_2^{-1}\sigma_1^{-1}} =v{\sigma_1^{-1}} = u.
$$

Since $\langle\sigma_2\rangle$ transitively permutes the $q$ neighbours of $v$ and $q \geq 3$, this implies that $u\not = w$. Consequently,
the orbit of $v$ under $\langle\sigma_1\rangle$ consists of at least three distinct vertices, implying that $p\ge 3$. Moreover,
since $\sigma_2\sigma_1\sigma_2 = \sigma_1^{-1}$, the fact that $\sigma_2$ is not an involution implies that $\sigma_2$ does not invert $\sigma_1$.

Next, since $\sigma_2\sigma_1$ also is an involution and $v(\sigma_2\sigma_1) = v\sigma_1 = w$, we see that $\sigma_2\sigma_1$ inverts the edge $\{v,w\}$.
This shows that the arc-stabiliser $\Stab_G(vw)$ has index $2$ in the edge-stabiliser $\Stab_G(e')$ where $e' = \{v,w\}$. 
Hence, if $|\mathrm{Stab}_G(vw)| =1$, then $\Stab_G(e') = \langle \sigma_2\sigma_1\rangle$.
 Observe also that the connectedness of $\Gamma$ and the fact that $\Stab_G(v)$ is transitive on the neighbourhood of $v$, together with
 the existence of an edge-reversing automorphism $\sigma_2\sigma_1$, implies that $G$ is transitive on the arcs of $\Gamma$.

Now suppose that there exists  $\rho\in \Aut(\Gamma)$ such that $v\rho=v$ and $\sigma_1^\rho=\sigma_1^{-1}$.
Let $H=\langle \sigma_1, \sigma_2, \rho \rangle$  (where possibly $H=G$). Note that $v(\sigma_2\rho) = v = v\rho^2$. Moreover, since $\rho$ inverts $\sigma_1$, it follows that $\sigma_1^{\rho^2} = \sigma_1$ and $\sigma_1^{\rho^{-1}} = \sigma_1^{-1}$, implying that
$$
w(\sigma_2\rho) = (w\sigma_2)\rho = u\rho = v\sigma_1^{-1}\rho = v\rho\sigma_1=v\sigma_1 = w\quad \text{and}\quad w\rho^2 = v\sigma_1\rho^2 = v\rho^2\sigma_1 = v\sigma_1 =  w.
$$
Therefore, $\sigma_2\rho, \rho^2 \in \Stab_H(vw)$. Since both $\rho$ and $\rho^{-1}$ invert $\sigma_1$ but $\sigma_2$ does not, we see that $\sigma_2\rho  \not = \rho^2$ and $\sigma_2\rho  \not = 1$.


If the hypothesis of  (b) holds, then $\rho\in G$ and thus $H=G$. Moreover, $\sigma_2\rho$ is the unique non-trivial element of $\mathrm{Stab}_G(vw)$
and therefore $\rho^2=1$, showing that $\Stab_G(e') = \langle \sigma_2\sigma_1, \sigma_2\rho\rangle 
= \langle \sigma_2\sigma_1, \rho\sigma_1\rangle$.
 We shall use these facts throughout the rest of the proof whenever we assume that the hypothesis of (b) holds.

Let $v_i = v\sigma_1^i$ for $i\in \ZZ$ and consider the sequence of vertices $f=(v_0, v_1, \ldots, v_{p-1}, v_p)$. 
Clearly, $v_0=v=v_p$,
$w=v_1$ and $u=v_{p-1}$.
Since $w$ is adjacent to $v$ by assumption, the above sequence is a closed walk in $\Gamma$. Let us now show that under the hypothesis of (a) or (b),
the walk $f$ is in fact a cycle.
If this were not the case, then $v=v\sigma_1^k$ for some $k\in \{1,\ldots, p-1\}$ and thus  $\sigma_1^k \in \Stab_G(vw)$. 
Since $\sigma_1$ is of order $p$, only the hypothesis of (b) can hold, in which case $\sigma_1^k$ is the unique non-trivial element of  $\Stab_G(vw)$ and thus $\sigma_1^k = \sigma_2\rho$.
But then $\sigma_2 = \sigma_1^{k}\rho^{-1}$, contradicting the fact that $\sigma_2$ does not invert $\sigma_1$ while $\rho^{-1}$ does. This shows that $f$ is indeed a cycle of $\G$.

Let us now consider the action of $\rho$ on the vertices of $f$, where $\rho \in \Aut(\G)$, $v\rho = v$ and $\sigma_1^\rho = \sigma_1^{-1}$, as in (a) and (b). Let $i\in \{0, \ldots, p-1\}$. Then $v_i\rho = v\sigma_1^i\rho = 
v\rho\sigma_1^{-i} = v\sigma_1^{-i} = v_{p-i}$. In short, $\rho$ preserves $f$ and acts on it as a reflection through $v$.

Now let $\cC =  \{f\gamma \mid \gamma \in G\}$ be the orbit of the cycle $f$ under $G$. We shall show that $\cC$ is a double cycle cover of $\Gamma$
and that the vertex figure with respect to $\cC$ at every vertex is connected. Since $G$ is edge- and vertex-transitive and $\cC$ is an orbit under the action of  $G$,
it suffices to show that the edge $e' = \{v,w\}$ belongs to exactly two cycles in $\cC$ and that the vertex-figure at $v$ is connected.

Suppose that the edge $e'$ lies on a cycle $f' \in \cC$. 
Choose $\alpha\in G$ such that $f' \alpha = f$. 
Since $\langle \sigma_1\rangle$
is transitive on the edges of $f$, we see that for an appropriate $i\in \ZZ_p$ the element $\gamma = \alpha\sigma_1^i \in G$ maps $f'$ to $f$ and preserves $e'$.
That is, $f'\gamma=f$ and $\gamma\in \Stab_G(e')$.
Now recall that $\Stab_G(e')$ is either $\langle \sigma_2\sigma_1\rangle$ (in case that the hypothesis of (a) holds) or $\langle \sigma_2\sigma_1, \rho\sigma_1\rangle$ (in the case of (b)).
Since both $\rho$ and $\sigma_1$ preserve $f$ and since $\gamma \in \Stab_G(e')$, this implies that $f' \in \{f,f\sigma_2\sigma_1\}$. In particular,
$e'$ belongs to at most two cycles in $\cC$, namely $f$ and $f\sigma_2\sigma_1$, and it belongs to exactly two cycles in $\cC$ unless $f=f\sigma_2\sigma_1$.
If the latter happens, then $f = f\sigma_2\sigma_1 = f\sigma_1^{-1} \sigma_2^{-1} = f\sigma_2^{-1}$, implying that $(u,v,w)\sigma_2^{-1} = (w,v,w\sigma_2^{-1})$
 is a path of length $2$ of $f$ centred at $v$. But then $w\sigma_2^{-1} = u = w\sigma_2$, contradicting the assumption that $\langle \sigma_2\rangle$ transitively permutes the $q$ neighbours of $v$ and that $q \geq 3$. 
 This shows that $e'$ lies in precisely two cycles in $\cC$, and thus, that $\cC$ is a cycle double cover of $\Gamma$.
 
Now consider the vertex figure at $v$ with respect to $\cC$. For $i\in \ZZ_q$, let $z_i = w\sigma_2^i$ and observe that
 $\{z_i \mid i \in \ZZ_q\}$ is the neighbourhood of $v$ in $\Gamma$. Moreover,  the edges
 $\{z_{0},v\} = \{w,v\}$ and $\{z_1,v\} = \{u,v\}$ are two consecutive edges on a cycle in $\cC$, implying that they are adjacent in the vertex figure at $v$.
 But then the edges $\{w,v\}\sigma_2^i = \{z_i, v\}$ and $\{u,v\}\sigma_2^i = \{z_{i+1},v\}$ are also adjacent in the vertex figure, implying that the vertex figure
 is connected. 
 
 We are now in a position to apply Lemma~\ref{lemma:polyhedron} to conclude that the graph $\G$ is the skeleton of a map $\m$ with $\Aut^+(\m)=G$.
The existence of $\sigma_1$ and $\sigma_2$ implies that the map $\m$ is rotary. 
Now recall that whenever $\rho$ exists (as in (a) or (b)), it preserves $f$ and normalizes $G$, implying that it preserves $\mathcal{C}$ and is thus an automorphism of the map $\m$. Moreover, the automorphism $\sigma_2\rho$ is a non-trivial element of the group $\Stab_{H}(vw)$ where $H = \langle \sigma_1,\sigma_2,\rho\rangle \le \Aut(\m)$. In particular, $\sigma_2\rho$ is an involution (and so is $\rho$), implying that $\sigma_2^{\rho} =  \sigma_2^{-1}$. 
The rest of the claims of the theorem now follow directly from Lemma~\ref{chiralorreflexible}.
%
\end{proof}

We conclude the section with a lemma that, given a polytopal reflexible map $\m$, allows us to obtain yet another map with the same skeleton,  often referred to as the Petrie dual of $\m$ (see for example \cite{coxeter2013generators}).

\begin{lemma}
\label{lem:Petrie}
Let $\m$ be a polytopal reflexible map of type $\{p,q\}$ with $q\geq 3$. 
Suppose that the skeleton $\Gamma$ of $\m$ is simple. 
Let $\sigma_1, \sigma_2$ be the distinguished generators of $\Aut^+(\m)$ with respect to some base flag $\Phi$, and let $\rho$ be an involutory automorphism of $\m$ fixing the vertex $v$ of $\Phi$ and inverting $\sigma_1$ and $\sigma_2$.
Then there exists a reflexible polytopal map $\m^\pi$ with the skeleton $\Gamma$ such that $\eta_1:=\sigma_1\sigma_2\rho$ and $\eta_2:=\sigma_2$ are the distinguished generators of $\Aut^+(\m^\pi)$ with respect to some base flag. The map $\m^\pi$ is non-orientable if and only if $\rho \in \langle \eta_1,\eta_2\rangle$.
\end{lemma}

\begin{proof}
We need to show that the conditions for $\sigma_1$ and $\sigma_2$ of Theorem~\ref{prop:undergraph} are satisfied by $\eta_1$ and $\eta_2$.
Condition (i) of Theorem~\ref{prop:undergraph} is trivially satisfied by $\eta_2$.
Next, since $\sigma_1\sigma_2$ is an involution and $\rho$ inverts $\sigma_1$, we see that $v\eta_1 = v\sigma_1\sigma_2\rho = v\sigma_2^{-1}\sigma_1^{-1}\rho = v\rho\sigma_1 = v\sigma_1$, showing that condition (ii) of Theorem~\ref{prop:undergraph} is satisfied by $\eta_1$.
Since $\rho$ also inverts $\sigma_2$, we see that $(\eta_1\eta_2)^2 = (\sigma_1\sigma_2\rho\sigma_2)^2 = (\sigma_1\rho)^2 = 1$, implying that condition (iii) of Theorem~\ref{prop:undergraph} is  satisfied by $\eta_1$ and $\eta_2$.

It is also straightforward to see that $\rho$ inverts $\eta_1$: $\eta_1^\rho = \rho\sigma_1\sigma_2 = \rho \sigma_2^{-1}\sigma_1^{-1}=(\sigma_1\sigma_2\rho)^{-1}=\eta_1^{-1}$. Set $H = \langle \eta_1, \eta_2\rangle$ and observe that $H \leq \Aut(\m)$, implying that the arc stabiliser of an arc in the group $H$ has order at most $2$. In fact, since clearly $H = \Aut(\m)$ if and only if $\rho \in H$, the arc stabiliser of an arc in the group $H$ has order $2$ or $1$, depending on whether $\rho \in H$ or not, respectively. Theorem~\ref{prop:undergraph} thus implies that $\eta_1$ and $\eta_2$ are the distinguished generators (with respect to some base flag) of a reflexible polytopal map $\m^\pi$ with skeleton $\G$, and that $\m^\pi$ is non-orientable if and only if $\rho \in H$.
\end{proof}

\section{The graphs $C_n[mK_1]$}
\label{sec:thegraphs}

Let $m$ and $s$ be positive integers where $m \geq 3$ is odd, and let $n = ms$. 
We let $\G = C_n[mK_1]$ be the lexicographic product of the $n$-cycle $C_n$ by the edgeless graph $mK_1$ of order $m$. 
That is, the vertices of $\G$ are of the form $(v,w)$, where $v\in C_n$ and $w\in mK_1$, and there is an edge between the vertices $(v_1,w_1)$ and $(v_2, w_2)$ whenever there is an edge between $v_1$ and $v_2$ in $C_n$.
Hence, $\G$ is a graph with $mn$ vertices and each vertex has valency $2m$, implying that $\G$ has $2m^2 n$ arcs.

By labelling the vertices of $C_n$ and $mK_1$ appropriately, we may assume that the vertex set of $C_n[mK_1]$ is the Cartesian product $\{1,\ldots,n\} \times \{1,\ldots,m\}$
with the vertices $(i_1,j_1)$ and $(i_2,j_2)$ adjacent  if and only if $i_1=i_2\pm 1$, where the sum is taken mod $n$ and where we use $n$ instead of $0$ (this convention is taken throughout the paper).

It is well known and easy to see that the automorphism group $\Aut(\G)$ of $\G$ is the wreath product $S_m \wr D_n$ of the symmetric group $S_m$ by the dihedral group $D_n$ (recall that $n\neq 4$). In other words, it is 
equal to the semidirect product
$$
	(S_m \times S_m \times  \cdots \times S_m) \rtimes D_n
$$
in its imprimitive action, 
where we have $n$ copies of $S_m$ in the direct product and where $D_n$ acts on this direct product by permuting the coordinates. 
We write the elements of this group in the form $(\alpha_1, \alpha_2, \ldots , \alpha_n)x$, where $\alpha_i \in S_m$ for each $i$, $1 \leq i \leq n$, and $x \in D_n$. 
The action of $\Aut(\G)$ on the vertex-set of $\G$ is then given by
\begin{equation}
(i,j)(\alpha_1, \alpha_2, \ldots , \alpha_n)x = (ix,j\alpha_i)
\end{equation}
for every vertex $(i,j) \in \{1,\ldots,n\} \times \{1,\ldots, m\} $ and $(\alpha_1, \alpha_2, \ldots , \alpha_n)x\in S_m \wr D_n$.

Let $c \in S_m$ be the $m$-cycle $c = (1\,2\,\cdots\,m)$ and let $t \in S_m$ be the involution fixing $1$ and interchanging each $j$, $2 \leq j \leq m$, with $m-j+2$. Furthermore let $r \in D_n$ be the $n$-cycle $r = (1\,2\,\cdots \,n)$ and let $z \in D_n$ be the reflection fixing $1$ and interchanging each $j$, $2 \leq j \leq n$, with $n-j+2$. We point out that for each $(\alpha_1, \alpha_2, \ldots , \alpha_n) \in S_m \times S_m \times  \cdots \times S_m$ the following holds:
\begin{eqnarray}
\label{eq:rz}
	r(\alpha_1, \alpha_2, \ldots , \alpha_n) &= &(\alpha_2,\alpha_3, \ldots , \alpha_n,\alpha_1)r \nonumber \\ &\mathrm{and}&\\
	z(\alpha_1, \alpha_2, \ldots , \alpha_n) &=& (\alpha_1, \alpha_n, \alpha_{n-1}, \ldots , \alpha_3, \alpha_2)z. \nonumber 
\end{eqnarray}

\section{Chiral maps}
\label{sec:chiral}

In this section we construct a family of chiral maps with the skeletons $C_n[mK_1]$, where $m\geq 3$ is an odd integer and $n=sm$, with $s$ a positive integer. 
Throughout this section, we let $\G$ denote the graph $C_n[mK_1]$ and let the elements $c,t,r,z \in \Aut(\G)$ be as in Section~\ref{sec:thegraphs}.
We start by the following definitions and a lemma:
\begin{eqnarray}
\label{eq:chiral}
        \ell_i & = & 1+1+2+3+\cdots + (i-2) = 1 + \frac{(i-2)(i-1)}{2} \hbox{ for each } i\in \{2,3,\ldots,n\},  \nonumber \\  
	\sigma_1 &=& (c,1,1,\ldots ,1)r,   \nonumber  \\
	 \sigma_2 &=& (t,tc^{\ell_2},tc^{\ell_3},tc^{\ell_4},\ldots , tc^{\ell_n})z\> = \> (t,tc,tc^2,tc^4,\ldots,tc^{1+(n-2)(n-1)/2})z,\ \text{and}\\ 
	 G &=& \langle \sigma_1, \sigma_2 \rangle. \nonumber
\end{eqnarray}

\begin{lemma}
\label{lemma:rotary}
Let $\sigma_1$, $\sigma_2$ and $G$ be as in~\eqref{eq:chiral}. Then the following holds:
\begin{equation*}
|\sigma_1| = mn, \hspace{1cm} |\sigma_2| = 2m, \hspace{1cm} |\sigma_1\sigma_2| = 2 \hspace{1cm}\text{and} \hspace{1cm} |G| = 2m^2 n.
\end{equation*}

\end{lemma}

\begin{proof}
    Using~\eqref{eq:rz} we first note that:
\begin{equation}
\label{eq:sig1n}
	\sigma_1^n = (c,c,c,\ldots , c)1,
\end{equation}
which implies that the order of $\sigma_1^n$ is the same as that of $c$, and therefore the order of $\sigma_1$ is $mn$.
Observe that for each $i$ with $3 \leq i \leq n$ one of $n$ and $n+3-2i$ is even. Since $m$ is odd and $n = ms$, the product $n(n+3-2i)$ is an even multiple of $m$, implying that 
$$
	\ell_{n-i+3} = 1+\frac{(n-i+1)(n-i+2)}{2} = 1 + \frac{n(n+3-2i) + (i-1)(i-2)}{2} \equiv \ell_i \pmod{m}.
$$
Therefore, 
\begin{equation}
\label{eq:sigma2}
	\sigma_2 = (t,tc,tc^2,tc^4,tc^7,\ldots , tc^7,tc^4,tc^2)z.
\end{equation}
Using~\eqref{eq:rz} again it now easily follows that 
\begin{equation}
\label{eq:sig2squared}
	\sigma_2^2 = (1,c,c^2,c^3,\ldots, c^{n-1})1.
\end{equation}
In particular, $\sigma_2$ is of order $2m$. 
By~\eqref{eq:rz}, ~\eqref{eq:chiral} and \eqref{eq:sigma2}, it follows that
$$
	\sigma_1\sigma_2 = (t,tc^2,tc^4,\ldots , tc^4, tc^2, t)rz,
$$ 
and one then easily verifies that this is an involution, settling the first three claims of the lemma.

We shall now determine the order of the group $G$. To do so, start by noticing that from~\eqref{eq:chiral} and \eqref{eq:sig2squared} we also easily deduce that
$$
	\sigma_2^{-2}\sigma_1\sigma_2^2 = (c^2,c,c,\ldots,c)r = \sigma_1^{n+1} \in \langle \sigma_1 \rangle.
$$
Therefore, since the intersection $\langle \sigma_1 \rangle \cap \langle \sigma_2 \rangle$ clearly coincides with the intersection
$\langle \sigma_1^n \rangle \cap \langle \sigma_2^2 \rangle$, which by~\eqref{eq:sig1n} and~\eqref{eq:sig2squared} is trivial, this shows that the subgroup $H = \langle \sigma_1, \sigma_2^2 \rangle$ is of order $m^2n$. 
Now, since $\sigma_1\sigma_2$ is an involution, so is $\sigma_2\sigma_1$, and thus $\sigma_2\sigma_1 = \sigma_1^{-1}\sigma_2^{-1}$. 
Therefore
$$
	(\sigma_2\sigma_1) \sigma_1 (\sigma_2\sigma_1) = \sigma_1^{-1} \sigma_2^{-2} \in H
$$
and
$$
	(\sigma_2\sigma_1) \sigma_2^2(\sigma_2\sigma_1) = \sigma_1^{-1} \sigma_2^2 \sigma_1 \in H,
$$
implying that $\sigma_2\sigma_1$ normalises $H$. Since $\langle \sigma_1, \sigma_2 \rangle = \langle \sigma_1, \sigma_2^2, \sigma_2\sigma_1 \rangle$, this finally shows that the group $G$ has order $2m^2n$. 
\end{proof}

The above lemma has the following immediate consequence.

\begin{proposition}
\label{pro:rotary}
Let $m\geq 3$ be an odd integer and let $n=sm$, where $s$ is a positive integer. 
  Let $\sigma_1$, $\sigma_2$ and $G$ be as in~\eqref{eq:chiral}. Then $\G=C_n[mK_1]$ is the skeleton of a polytopal chiral map $\m$  of type $\{mn, 2m\}$ with $\Aut(\m) = \Aut^+(\m)=G$. 
\end{proposition}

\begin{proof}
Let $v:=(1,1)$ and note that $v\sigma_2 = (1,1)\sigma_2 = (1z,1t) = (1,1)$, showing that $\sigma_2$ fixes $v$.
 Observe that the neighbourhood of $v$ consists of the vertices of the form $(2,j)$ and $(n,k)$, for $1\leq j,k \leq m$. 
    Note that $(2,j)\sigma_2 = (2z, jtc)=(n, m-j+3)$ and $(n,k)\sigma_2 = (nz, ktc^2) = (2, m-k+4)$, where the second component is computed modulo $m$ if necessary. 
    Therefore, $ \sigma_2 $ cyclically permutes the vertices $(2,2), (n,1), (2,3), (n,m), (2,4), (n,m-1), (2,5), \dots, (n,2)$ around $v$, implying that $\langle \sigma_2 \rangle$ acts transitively on the neighbours of $v$. Hence, $\sigma_2$ satisfies condition (i) of Theorem~\ref{prop:undergraph}.

Further, $v\sigma_1 = (2,2)$, which is a neighbour of $v$ (in fact, the cycle through $v$ induced by $\sigma_1$ is 
$((1,1) , (2,2), (3,2), \dots, (n,2), (1,2), (2,3), \dots, (1,3), (2,4), \dots, (n,1))$, and is depicted in Figure~\ref{fig:chiralface}), showing that $\sigma_1$ satisfies condition (ii) of Theorem~\ref{prop:undergraph}. 
\begin{figure}
    \centering
    \newrgbcolor{qqwuqq}{0. 0.4 0.}
\psset{xunit=2.0cm,yunit=1.5cm,algebraic=true,dimen=middle,dotstyle=o,dotsize=5pt 0,linewidth=1.6pt,arrowsize=3pt 2,arrowinset=0.25}
\begin{pspicture*}(2.9305001166246356,-5.560868898641911)(10.13024663923072,-2.236422182163391)
\psline[linewidth=2.pt,linecolor=qqwuqq](4.,-3.)(5.,-3.)
\psline[linewidth=2.pt,linecolor=qqwuqq](5.,-3.)(6.,-3.)
\psline[linewidth=1.pt](6.,-3.)(7.,-3.)
\psline[linewidth=2.pt,linecolor=qqwuqq](7.,-3.)(8.,-3.)
\psline[linewidth=2.pt,linecolor=qqwuqq](8.,-3.)(9.,-3.)
\psline[linewidth=2.pt,linecolor=qqwuqq](9.,-3.)(9.64385302515989,-3.0030743857047866)
\psline[linewidth=2.pt,linecolor=qqwuqq](4.,-3.)(3.4813532062736874,-3.0030743857047866)
\psline[linewidth=2.pt,linecolor=qqwuqq](4.,-4.)(5.,-4.)
\psline[linewidth=2.pt,linecolor=qqwuqq](4.,-4.)(3.4813532062736874,-4.003074385704787)
\psline[linewidth=2.pt,linecolor=qqwuqq](5.,-4.)(6.,-4.)
\psline[linewidth=1.pt](6.,-4.)(7.,-4.)
\psline[linewidth=2.pt,linecolor=qqwuqq](7.,-4.)(8.,-4.)
\psline[linewidth=2.pt,linecolor=qqwuqq](8.,-4.)(9.,-4.)
\psline[linewidth=2.pt,linecolor=qqwuqq](9.,-4.)(9.64385302515989,-4.003074385704787)
\psline[linewidth=2.pt,linecolor=qqwuqq](4.,-5.)(3.4813532062736874,-5.003074385704787)
\psline[linewidth=2.pt,linecolor=qqwuqq](4.,-5.)(5.,-5.)
\psline[linewidth=2.pt,linecolor=qqwuqq](5.,-5.)(6.,-5.)
\psline[linewidth=1.pt](6.,-5.)(7.,-5.)
\psline[linewidth=2.pt,linecolor=qqwuqq](7.,-5.)(8.,-5.)
\psline[linewidth=2.pt,linecolor=qqwuqq](8.,-5.)(9.,-5.)
\psline[linewidth=2.pt,linecolor=qqwuqq](9.,-5.)(9.64385302515989,-5.003074385704787)
\psline[linewidth=1.pt](4.,-3.)(3.5047403023605423,-3.4591227593984533)
\psline[linewidth=1.pt](4.,-4.)(3.4930467543171146,-3.657913076136718)
\psline[linewidth=1.pt](4.,-3.)(3.457966110186833,-3.9268646811355468)
\psline[linewidth=1.pt](4.,-4.)(3.5047403023605423,-4.459122759398453)
\psline[linewidth=1.pt](4.,-5.)(3.4930467543171146,-4.657913076136718)
\psline[linewidth=1.pt](4.,-3.)(5.,-4.)
\psline[linewidth=1.pt](5.,-4.)(6.,-3.)
\psline[linewidth=2.pt,linecolor=qqwuqq](6.,-3.)(7.,-4.)
\psline[linewidth=1.pt](7.,-4.)(8.,-3.)
\psline[linewidth=1.pt](8.,-3.)(9.,-4.)
\psline[linewidth=1.pt](9.,-4.)(8.,-5.)
\psline[linewidth=1.pt](8.,-5.)(7.,-4.)
\psline[linewidth=1.pt](7.,-4.)(6.,-5.)
\psline[linewidth=1.pt](6.,-5.)(5.,-4.)
\psline[linewidth=1.pt](5.,-4.)(4.,-5.)
\psline[linewidth=1.pt](4.,-5.)(5.,-3.)
\psline[linewidth=1.pt](5.,-3.)(6.,-5.)
\psline[linewidth=2.pt,linecolor=qqwuqq](6.,-5.)(7.,-3.)
\psline[linewidth=1.pt](7.,-3.)(8.,-5.)
\psline[linewidth=1.pt](8.,-5.)(9.,-3.)
\psline[linewidth=1.pt](9.,-3.)(8.,-4.)
\psline[linewidth=1.pt](8.,-4.)(7.,-3.)
\psline[linewidth=1.pt](7.,-3.)(6.,-4.)
\psline[linewidth=1.pt](6.,-4.)(5.,-3.)
\psline[linewidth=1.pt](5.,-3.)(4.,-4.)
\psline[linewidth=1.pt](4.,-4.)(5.,-5.)
\psline[linewidth=1.pt](5.,-5.)(6.,-4.)
\psline[linewidth=2.pt,linecolor=qqwuqq](6.,-4.)(7.,-5.)
\psline[linewidth=1.pt](7.,-5.)(8.,-4.)
\psline[linewidth=1.pt](8.,-4.)(9.,-5.)
\psline[linewidth=1.pt](9.,-5.)(8.,-3.)
\psline[linewidth=1.pt](8.,-3.)(7.,-5.)
\psline[linewidth=1.pt](7.,-5.)(6.,-3.)
\psline[linewidth=1.pt](6.,-3.)(5.,-5.)
\psline[linewidth=1.pt](5.,-5.)(4.,-3.)
\psline[linewidth=1.pt](9.,-3.)(9.495259697639458,-3.459122759398453)
\psline[linewidth=1.pt](9.,-3.)(9.542033889813167,-3.9268646811355463)
\psline[linewidth=1.pt](9.,-4.)(9.506953245682887,-3.6579130761367176)
\psline[linewidth=1.pt](9.,-4.)(9.495259697639458,-4.459122759398453)
\psline[linewidth=1.pt](9.,-5.)(9.506953245682887,-4.657913076136718)
\rput[tl](5.8,-2.6){(1,1)}
\rput[tl](5.8,-4.2272245763336675){(1,2)}
\rput[tl](5.8,-5.212961684126717){(1,3)}
\rput[tl](6.8,-2.6){(2,1)}
\psline[linewidth=1.pt](4.,-5.)(3.457966110186833,-4.073135318864454)
\psline[linewidth=1.pt](8.990450256537116,-4.99995440016023)(9.541310749621408,-4.078308268501542)
\begin{scriptsize}
\psdots[dotstyle=*](5.,4.)
\psdots[dotstyle=*](5.5,3.1339745962155616)
\psdots[dotstyle=*](6.5,3.1339745962155616)
\psdots[dotstyle=*](7.,4.)
\psdots[dotstyle=*](6.5,4.866025403784438)
\psdots[dotstyle=*](5.5,4.866025403784438)
\psdots[dotstyle=*](4.,4.)
\psdots[dotstyle=*](5.,5.732050807568877)
\psdots[dotstyle=*](4.,4.)
\psdots[dotstyle=*](7.,5.732050807568877)
\psdots[dotstyle=*](5.,5.732050807568877)
\psdots[dotstyle=*](8.,4.)
\psdots[dotstyle=*](7.,5.732050807568877)
\psdots[dotstyle=*](7.,2.2679491924311233)
\psdots[dotstyle=*](8.,4.)
\psdots[dotstyle=*](5.,2.2679491924311233)
\psdots[dotstyle=*](7.,2.2679491924311233)
\psdots[dotstyle=*](4.,4.)
\psdots[dotstyle=*](5.,2.2679491924311233)
\psdots[dotstyle=*](4.5,6.598076211353316)
\psdots[dotstyle=*](3.,4.)
\psdots[dotstyle=*](7.5,6.598076211353316)
\psdots[dotstyle=*](4.5,6.598076211353316)
\psdots[dotstyle=*](9.,4.)
\psdots[dotstyle=*](7.5,6.598076211353316)
\psdots[dotstyle=*](7.5,1.401923788646684)
\psdots[dotstyle=*](9.,4.)
\psdots[dotstyle=*](4.5,1.401923788646684)
\psdots[dotstyle=*](7.5,1.401923788646684)
\psdots[dotstyle=*](3.,4.)
\psdots[dotstyle=*](4.5,1.401923788646684)
\psdots[dotstyle=*](4.,-3.)
\psdots[dotstyle=*](5.,-3.)
\psdots[dotstyle=*](6.,-3.)
\psdots[dotstyle=*](7.,-3.)
\psdots[dotstyle=*](8.,-3.)
\psdots[dotstyle=*](9.,-3.)
\psdots[dotstyle=*](4.,-4.)
\psdots[dotstyle=*](4.,-4.)
\psdots[dotstyle=*](5.,-4.)
\psdots[dotstyle=*](4.,-4.)
\psdots[dotstyle=*](5.,-4.)
\psdots[dotstyle=*](6.,-4.)
\psdots[dotstyle=*](6.,-4.)
\psdots[dotstyle=*](7.,-4.)
\psdots[dotstyle=*](7.,-3.)
\psdots[dotstyle=*](8.,-3.)
\psdots[dotstyle=*](7.,-3.)
\psdots[dotstyle=*](8.,-3.)
\psdots[dotstyle=*](7.,-4.)
\psdots[dotstyle=*](8.,-4.)
\psdots[dotstyle=*](8.,-4.)
\psdots[dotstyle=*](9.,-4.)
\psdots[dotstyle=*](9.,-4.)
\psdots[dotstyle=*](4.,-5.)
\psdots[dotstyle=*](4.,-5.)
\psdots[dotstyle=*](5.,-5.)
\psdots[dotstyle=*](5.,-5.)
\psdots[dotstyle=*](6.,-5.)
\psdots[dotstyle=*](6.,-5.)
\psdots[dotstyle=*](7.,-5.)
\psdots[dotstyle=*](7.,-5.)
\psdots[dotstyle=*](8.,-5.)
\psdots[dotstyle=*](8.,-5.)
\psdots[dotstyle=*](9.,-5.)
\psdots[dotstyle=*](9.,-5.)
\psdots[dotstyle=*](4.,-4.)
\psdots[dotstyle=*](4.,-4.)
\psdots[dotstyle=*](4.,-4.)
\psdots[dotstyle=*](4.,-4.)
\psdots[dotstyle=*](4.,-5.)
\psdots[dotstyle=*](9.,-3.)
\psdots[dotstyle=*](9.,-3.)
\psdots[dotstyle=*](9.,-4.)
\psdots[dotstyle=*](9.,-4.)
\psdots[dotstyle=*](9.,-5.)
\psdots[dotstyle=*](9.,-4.)
\end{scriptsize}
\end{pspicture*}
    \caption{The graph $C_6[3K_1]$ and the base face of the chiral map constructed in Proposition~\ref{pro:rotary}.}
    \label{fig:chiralface}
\end{figure}
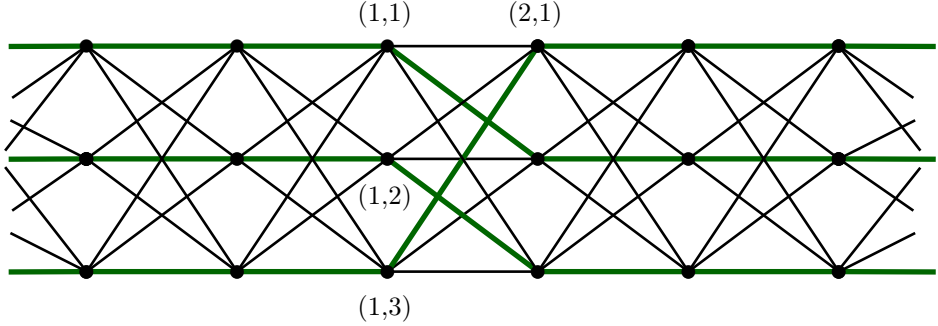

Moreover, as shown in Lemma~\ref{lemma:rotary}, $(\sigma_1\sigma_2)^2=1$, implying that condition (iii) of Theorem~\ref{prop:undergraph} is fulfilled. 
Theorem~\ref{prop:undergraph} then implies that $G$ is arc-transitive. Since $|G|=2m^2n$, which equals the number of arcs of $\G$, we see that
 the arc-stabiliser in $G$ is trivial. By part (a) of Theorem~\ref{prop:undergraph}, it follows that $\G$ is the skeleton of a polytopal map $\m$ 
of type $\{mn,2m\}$.

  Finally, to show that $\m$ is chiral, it suffices to show that $\Aut(\G)$ contains no involution $\rho$ fixing $v$, normalizing $G$ and inverting $\sigma_1$. By way of contradiction suppose such a $\rho$ exists and note that by Theorem~\ref{prop:undergraph} it inverts $\sigma_2$. Moreover, since it fixes $v$ and inverts $\sigma_1$ it must be of the form
$$
	\rho = (\alpha_1, \alpha_2, \ldots, \alpha_n)z.
$$
Since $\rho$ is an involution we deduce that
\begin{equation}
\label{eq:rho1}
	\alpha_1^2 = 1\ \text{and}\ \alpha_{n-i+2} = \alpha_i^{-1}\ \text{for}\ \text{all}\ i,\ 2 \leq i \leq n.
\end{equation}
By~\eqref{eq:chiral}, we see that $\rho \sigma_1 = (\alpha_1c, \alpha_2,\alpha_3, \ldots , \alpha_n)zr$, and so
$$
(\rho \sigma_1)^2 = (\alpha_1 c \alpha_2, \alpha_2 \alpha_1 c, \alpha_3 \alpha_n, \alpha_4\alpha_{n-1}, \ldots , \alpha_n\alpha_3)1. 
$$
Since $\rho$ inverts $\sigma_1$ we also deduce that $\alpha_1 c \alpha_2 = 1$ and $\alpha_{n-i+3} = \alpha_i^{-1}$ for all $i$ with $3 \leq i \leq n$. Together with~\eqref{eq:rho1}, this thus implies that
\begin{equation}
\label{eq:cond_rho}
	\alpha_1 c \alpha_2 = 1,\ \alpha_2 = \alpha_3 = \ldots = \alpha_n, \ \text{and}\ \alpha_2^2 = 1.
\end{equation}
Finally, by~\eqref{eq:sigma2} we now see that 
$$
	\rho \sigma_2 = (\alpha_1t, \alpha_2tc^2, \alpha_2tc^4, \alpha_2 tc^7, \ldots, \alpha_2 tc)1,
$$
which can thus only be an involution if $\alpha_2$ is centralised by $tc$ and $tc^2$. But then it is also centralised by $c$, and so the fact that the centraliser of $c$ in $S_m$ is $\langle c \rangle$ forces $\alpha_2 = 1$ (recall that $m$ is odd). Then~\eqref{eq:cond_rho} implies that $\alpha_1 c = 1$, which contradicts $\alpha_1^2 = 1$. This finally shows that no involutory automorphism $\rho$ fixing $v$, normalising $G$ and inverting $\sigma_1$ exists, and thus that the map $\m$ is chiral.
\end{proof}

\section{Reflexible maps}
\label{sec:reflexible}

In this section we construct non-orientable reflexible and orientable reflexible maps whose skeleton is $C_n[mK_1]$, with $n = sm$ where $m \geq 3$ is an odd integer and where $s$ is a positive integer not divisible by $4$. The constructions for $s$ odd and $s$ even are a bit different (even though there are many similarities). 

Throughout this section we let $\G = C_n[mK_1]$, where $m \geq 3$ is odd, and where $n = sm$ for some $s \geq 1$. We represent $\Aut(\G)$ just as we did in Section~\ref{sec:thegraphs} and we also let $c$, $t$, $r$ and $z$ have the same meaning as in Section~\ref{sec:thegraphs}. 

\subsection{The examples with $s$ odd}

Throughout this subsection let $s$ be odd. We first construct a non-orientable reflexible map. To this end, set
\begin{eqnarray}
\label{eq:ref_nor_s_odd}
	\sigma_1 &=& (1,1,\ldots , 1, tc^{-1})r,\nonumber \\ 
	\sigma_2 &=& (t,c^{-1},c^{3}, c^{-5}, \ldots , c^{(-1)^{i-1}(2i-3)}, \ldots, c^5, c^{-3})z,  \text{ and} \\
	G &=& \langle \sigma_1, \sigma_2 \rangle .
 \nonumber
\end{eqnarray}
Note that the assumptions that $n$ is odd and $c^n = (c^m)^s = 1$ indeed yield $c^{(-1)^{n-1}(2n-3)} = c^{-3}$.     

\begin{lemma}
\label{lemma:nonor_s_odd}
Let $\sigma_1$, $\sigma_2$ and $G$ be as in~\eqref{eq:ref_nor_s_odd}.
Then the following holds:
\begin{equation*}
|\sigma_1| = 2n, \hspace{1cm} |\sigma_2| = 2m, \hspace{1cm} |\sigma_1\sigma_2| = 2, \hspace{1cm}\text{and} \hspace{1cm} |G| = 4m^2 n.
\end{equation*}
\end{lemma}

\begin{proof}
We first note that 
\begin{equation}
\label{eq:nonor_s1}
	\sigma_1^n = (tc^{-1},tc^{-1},\ldots , tc^{-1})1,
\end{equation}
implying that $|\sigma_1| = 2n$. Similarly, 
$$
	\sigma_2^2 = (t,c^{-1},\ldots , c^{(-1)^{i-1}(2i-3)}, \ldots, c^{-3})(t,c^{-3},\ldots , c^{(-1)^{n-i+1}(2n-2i+1)}, \ldots , c^{-1})1,
$$
and so setting $\varphi = \sigma_2^2$ we see that 
\begin{equation}
\label{eq:nonor_varphi}
	 \varphi = (1, c^{-4}, c^8, \ldots , c^{(-1)^{i-1}4(i-1)}, \ldots, c^8, c^{-4})1.
\end{equation}
This of course implies that $|\sigma_2| = 2m$, as claimed. Next, note that
\begin{equation}
\label{eq:reflex_s1s2_s_odd}
	\sigma_1\sigma_2 = (c^{-1}, c^3, \ldots , c^{(-1)^i(2i-1)}, \ldots,  c^{-3}, c)rz,
\end{equation}
from which it is easy to verify that $\sigma_1\sigma_2$ is indeed an involution. 

Let $\psi = \sigma_1^{-1}\varphi \sigma_1$ and note that 
$$
	\psi = (tc^{-1},1,\ldots, 1)(c^{-4},1,c^{-4},c^8,\ldots , c^{(-1)^{i-2}4(i-2)}, \ldots, c^8)(tc^{-1},1,\ldots , 1)1,
$$
that is
\begin{equation}
\label{eq:nonor_psi}
	\psi = (c^{4},1,c^{-4},\ldots , c^{(-1)^{i}4(i-2)}, \ldots, c^8)1.
\end{equation}
Since $\varphi$ and $\psi$ commute and are both of order $m$ (recall that $m$ is odd), the group $H_1 = \langle \varphi, \psi \rangle \cong C_m \times C_m$ is abelian and is of order $m^2$. In a completely analogous way as we computed $\psi$ we can verify that 
$$
	\sigma_1^{-1}\psi \sigma_1 = (c^{-8},c^4,1,\ldots ,c^{(-1)^{i+1}4(i-3)}, \ldots,  c^{-12})1.
$$
It follows that $\sigma_1^{-1}\psi \sigma_1 = \varphi^{-1}\psi^{-2} \in H_1$, which implies that $\sigma_1$ normalizes $H_1$. Similarly, $\sigma_2$ commutes with $\varphi$ and since (recall that $\sigma_1\sigma_2$ is an involution)
\begin{equation}
\label{eq:sig2psi}
	\sigma_2^{-1}\psi \sigma_2 = \sigma_2^{-1}\sigma_1^{-1}\sigma_2^2\sigma_1\sigma_2 = \sigma_1\sigma_2^2\sigma_1^{-1} = \sigma_1\varphi\sigma_1^{-1}
\end{equation}
and $\sigma_1$ normalizes $H_1$, we see that in fact $\sigma_2$ also normalizes $H_1$, proving that $H_1$ is normal in $G$. Clearly, $\langle \sigma_1 \rangle \cap H_1 = 1$, and so $K_1 = \langle \sigma_1, \varphi, \psi \rangle$ is a group of order $2m^2 n$. Now, $\sigma_1$ normalizes $K_1$, and so the fact that 
\begin{equation}
\label{eq:sig2sig1_conj}
	\sigma_2^{-1}\sigma_1\sigma_2 = \sigma_2^{-1}\sigma_2^{-1}\sigma_1^{-1} = \varphi^{-1}\sigma_1^{-1} \in K_1
\end{equation}
implies that $K_1$ is also a normal subgroup of $G$. Since $\sigma_2 \notin K_1$ and $\sigma_2^2 \in K_1$, this finally proves that $G$ is of order $4m^2 n$, as claimed.
\end{proof}

\begin{proposition}
\label{pro:nonor_s_odd}
Let $s$ and $m$ be odd integers with $m\geq 3$.
  Let $\sigma_1$, $\sigma_2$ and $G$ be as in~\eqref{eq:ref_nor_s_odd}. Then $\G=C_n[mK_1]$ is the skeleton of a polytopal non-orientable reflexible map $\m$ of type $\{2n, 2m\}$ with $\Aut(\m) = \Aut^+(\m)=G$. 
\end{proposition}
\begin{proof}
Let $v:=(1,1)$ and observe that $\sigma_1$ maps $v$ to the neighbour $w:=(2,1)$ of $v$. In fact,  the orbit of $v$ under $\langle \sigma_1 \rangle$ constitutes the cycle $f$ of length $2n$
 with vertices $(1,1), (2,1),$ $ (3,1),$ $\ldots,$ $(n,1), (1,m),$ $(2,m), \ldots, (n,m)$, in that order (see Figure~\ref{fig:non-oreintable-sodd-face}).

\begin{figure}
    \centering
    \newrgbcolor{ududff}{0.30196078431372547 0.30196078431372547 1.}
\newrgbcolor{qqwuqq}{0. 0.39215686274509803 0.}
\psset{xunit=1.5cm,yunit=1.5cm,algebraic=true,dimen=middle,dotstyle=o,dotsize=5pt 0,linewidth=1.6pt,arrowsize=3pt 2,arrowinset=0.25}
\begin{pspicture*}(2.583761845572507,-5.7281201013688685)(13.783332719667442,-2.3080360925054935)
\psline[linewidth=2.pt,linecolor=qqwuqq](4.,-3.)(5.,-3.)
\psline[linewidth=1.pt](5.,-3.)(6.,-3.)
\psline[linewidth=2.pt,linecolor=qqwuqq](6.,-3.)(7.,-3.)
\psline[linewidth=2.pt,linecolor=qqwuqq](7.,-3.)(8.,-3.)
\psline[linewidth=2.pt,linecolor=qqwuqq](8.,-3.)(9.,-3.)
\psline[linewidth=2.pt,linecolor=qqwuqq](4.,-3.)(3.4813532062736874,-3.0030743857047866)
\psline[linewidth=1.pt](4.,-4.)(5.,-4.)
\psline[linewidth=1.pt](4.,-4.)(3.4813532062736874,-4.003074385704787)
\psline[linewidth=1.pt](5.,-4.)(6.,-4.)
\psline[linewidth=1.pt](6.,-4.)(7.,-4.)
\psline[linewidth=1.pt](7.,-4.)(8.,-4.)
\psline[linewidth=1.pt](8.,-4.)(9.,-4.)
\psline[linewidth=2.pt,linecolor=qqwuqq](4.,-5.)(3.4813532062736874,-5.003074385704787)
\psline[linewidth=2.pt,linecolor=qqwuqq](4.,-5.)(5.,-5.)
\psline[linewidth=1.pt](5.,-5.)(6.,-5.)
\psline[linewidth=2.pt,linecolor=qqwuqq](6.,-5.)(7.,-5.)
\psline[linewidth=2.pt,linecolor=qqwuqq](7.,-5.)(8.,-5.)
\psline[linewidth=2.pt,linecolor=qqwuqq](8.,-5.)(9.,-5.)
\psline[linewidth=1.pt](4.,-3.)(3.5047403023605423,-3.4591227593984533)
\psline[linewidth=1.pt](4.,-4.)(3.4930467543171146,-3.657913076136718)
\psline[linewidth=1.pt](4.,-3.)(3.457966110186833,-3.9268646811355468)
\psline[linewidth=1.pt](4.,-4.)(3.5047403023605423,-4.459122759398453)
\psline[linewidth=1.pt](4.,-5.)(3.4930467543171146,-4.657913076136718)
\psline[linewidth=1.pt](4.,-3.)(5.,-4.)
\psline[linewidth=1.pt](5.,-4.)(6.,-3.)
\psline[linewidth=1.pt](6.,-3.)(7.,-4.)
\psline[linewidth=1.pt](7.,-4.)(8.,-3.)
\psline[linewidth=1.pt](8.,-3.)(9.,-4.)
\psline[linewidth=1.pt](9.,-4.)(8.,-5.)
\psline[linewidth=1.pt](8.,-5.)(7.,-4.)
\psline[linewidth=1.pt](7.,-4.)(6.,-5.)
\psline[linewidth=1.pt](6.,-5.)(5.,-4.)
\psline[linewidth=1.pt](5.,-4.)(4.,-5.)
\psline[linewidth=1.pt](4.,-5.)(5.,-3.)
\psline[linewidth=2.pt,linecolor=qqwuqq](5.,-3.)(6.,-5.)
\psline[linewidth=1.pt](6.,-5.)(7.,-3.)
\psline[linewidth=1.pt](7.,-3.)(8.,-5.)
\psline[linewidth=1.pt](8.,-5.)(9.,-3.)
\psline[linewidth=1.pt](9.,-3.)(8.,-4.)
\psline[linewidth=1.pt](8.,-4.)(7.,-3.)
\psline[linewidth=1.pt](7.,-3.)(6.,-4.)
\psline[linewidth=1.pt](6.,-4.)(5.,-3.)
\psline[linewidth=1.pt](5.,-3.)(4.,-4.)
\psline[linewidth=1.pt](4.,-4.)(5.,-5.)
\psline[linewidth=1.pt](5.,-5.)(6.,-4.)
\psline[linewidth=1.pt](6.,-4.)(7.,-5.)
\psline[linewidth=1.pt](7.,-5.)(8.,-4.)
\psline[linewidth=1.pt](8.,-4.)(9.,-5.)
\psline[linewidth=1.pt](9.,-5.)(8.,-3.)
\psline[linewidth=1.pt](8.,-3.)(7.,-5.)
\psline[linewidth=1.pt](7.,-5.)(6.,-3.)
\psline[linewidth=2.pt,linecolor=qqwuqq](6.,-3.)(5.,-5.)
\psline[linewidth=1.pt](5.,-5.)(4.,-3.)
\rput[tl](5.787079966550198,-2.62114237500707){(1,1)}
\rput[tl](5.787079966550198,-4.198716336841937){(1,2)}
\rput[tl](5.8232076145311495,-5.186205381654601){(1,3)}
\rput[tl](6.786611560689854,-2.6090998256800866){(2,1)}
\psline[linewidth=1.pt](4.,-5.)(3.457966110186833,-4.073135318864454)
\psline[linewidth=2.pt,linecolor=qqwuqq](9.,-3.)(10.,-3.)
\psline[linewidth=2.pt,linecolor=qqwuqq](10.,-3.)(11.,-3.)
\psline[linewidth=2.pt,linecolor=qqwuqq](11.,-3.)(12.,-3.)
\psline[linewidth=1.pt](9.,-3.)(10.,-4.)
\psline[linewidth=1.pt](10.,-3.)(11.,-4.)
\psline[linewidth=1.pt](11.,-3.)(12.,-4.)
\psline[linewidth=1.pt](10.,-3.)(9.,-4.)
\psline[linewidth=1.pt](11.,-3.)(10.,-4.)
\psline[linewidth=1.pt](12.,-3.)(11.,-4.)
\psline[linewidth=1.pt](10.,-5.)(9.,-3.)
\psline[linewidth=1.pt](11.,-5.)(10.,-3.)
\psline[linewidth=1.pt](12.,-5.)(11.,-3.)
\psline[linewidth=1.pt](9.,-5.)(10.,-3.)
\psline[linewidth=1.pt](10.,-5.)(11.,-3.)
\psline[linewidth=1.pt](11.,-5.)(12.,-3.)
\psline[linewidth=1.pt](9.,-4.)(10.,-4.)
\psline[linewidth=1.pt](10.,-4.)(11.,-4.)
\psline[linewidth=1.pt](11.,-4.)(12.,-4.)
\psline[linewidth=2.pt,linecolor=qqwuqq](9.,-5.)(10.,-5.)
\psline[linewidth=2.pt,linecolor=qqwuqq](10.,-5.)(11.,-5.)
\psline[linewidth=2.pt,linecolor=qqwuqq](11.,-5.)(12.,-5.)
\psline[linewidth=1.pt](9.,-5.)(10.,-3.)
\psline[linewidth=1.pt](9.,-4.)(10.,-5.)
\psline[linewidth=1.pt](10.,-4.)(11.,-5.)
\psline[linewidth=1.pt](11.,-4.)(12.,-5.)
\psline[linewidth=1.pt](10.,-4.)(9.,-5.)
\psline[linewidth=1.pt](11.,-4.)(10.,-5.)
\psline[linewidth=1.pt](12.,-4.)(11.,-5.)
\psline[linewidth=2.pt,linecolor=qqwuqq](12.,-3.)(12.518646793726312,-3.003074385704786)
\psline[linewidth=1.pt](12.,-3.)(12.495259697639458,-3.459122759398453)
\psline[linewidth=1.pt](12.,-3.)(12.542033889813169,-3.9268646811355463)
\psline[linewidth=1.pt](12.,-4.)(12.506953245682887,-3.6579130761367176)
\psline[linewidth=1.pt](12.,-4.)(12.518646793726314,-4.003074385704786)
\psline[linewidth=1.pt](12.,-5.)(12.542033889813169,-4.073135318864453)
\psline[linewidth=1.pt](12.,-5.)(12.506953245682887,-4.657913076136717)
\psline[linewidth=2.pt,linecolor=qqwuqq](12.,-5.)(12.518646793726314,-5.003074385704786)
\begin{scriptsize}
\psdots[dotstyle=*](4.,-3.)
\psdots[dotstyle=*](5.,-3.)
\psdots[dotstyle=*](6.,-3.)
\psdots[dotstyle=*](7.,-3.)
\psdots[dotstyle=*](8.,-3.)
\psdots[dotstyle=*](9.,-3.)
\psdots[dotstyle=*](4.,-4.)
\psdots[dotstyle=*](4.,-4.)
\psdots[dotstyle=*](5.,-4.)
\psdots[dotstyle=*](4.,-4.)
\psdots[dotstyle=*](5.,-4.)
\psdots[dotstyle=*](6.,-4.)
\psdots[dotstyle=*](6.,-4.)
\psdots[dotstyle=*](7.,-4.)
\psdots[dotstyle=*](7.,-3.)
\psdots[dotstyle=*](8.,-3.)
\psdots[dotstyle=*](7.,-3.)
\psdots[dotstyle=*](8.,-3.)
\psdots[dotstyle=*](7.,-4.)
\psdots[dotstyle=*](8.,-4.)
\psdots[dotstyle=*](8.,-4.)
\psdots[dotstyle=*](9.,-4.)
\psdots[dotstyle=*](9.,-4.)
\psdots[dotstyle=*](4.,-5.)
\psdots[dotstyle=*](4.,-5.)
\psdots[dotstyle=*](5.,-5.)
\psdots[dotstyle=*](5.,-5.)
\psdots[dotstyle=*](6.,-5.)
\psdots[dotstyle=*](6.,-5.)
\psdots[dotstyle=*](7.,-5.)
\psdots[dotstyle=*](7.,-5.)
\psdots[dotstyle=*](8.,-5.)
\psdots[dotstyle=*](8.,-5.)
\psdots[dotstyle=*](9.,-5.)
\psdots[dotstyle=*](9.,-5.)
\psdots[dotstyle=*](4.,-4.)
\psdots[dotstyle=*](4.,-4.)
\psdots[dotstyle=*](4.,-4.)
\psdots[dotstyle=*](4.,-4.)
\psdots[dotstyle=*](4.,-5.)
\psdots[dotstyle=*](9.,-3.)
\psdots[dotstyle=*](9.,-3.)
\psdots[dotstyle=*](9.,-4.)
\psdots[dotstyle=*](9.,-4.)
\psdots[dotstyle=*](9.,-5.)
\psdots[dotstyle=*](9.,-4.)
\psdots[dotstyle=*](4.,-5.)
\psdots[dotstyle=*](10.,-3.)
\psdots[dotstyle=*](10.,-4.)
\psdots[dotstyle=*](10.,-5.)
\psdots[dotstyle=*](11.,-3.)
\psdots[dotstyle=*](11.,-4.)
\psdots[dotstyle=*](11.,-5.)
\psdots[dotstyle=*](12.,-3.)
\psdots[dotstyle=*](12.,-4.)
\psdots[dotstyle=*](12.,-5.)
\psdots[dotstyle=*](9.,-3.)
\psdots[dotstyle=*](10.,-3.)
\psdots[dotstyle=*](10.,-3.)
\psdots[dotstyle=*](11.,-3.)
\psdots[dotstyle=*](11.,-3.)
\psdots[dotstyle=*](12.,-3.)
\psdots[dotstyle=*](9.,-3.)
\psdots[dotstyle=*](10.,-4.)
\psdots[dotstyle=*](10.,-3.)
\psdots[dotstyle=*](11.,-4.)
\psdots[dotstyle=*](11.,-3.)
\psdots[dotstyle=*](12.,-4.)
\psdots[dotstyle=*](10.,-3.)
\psdots[dotstyle=*](9.,-4.)
\psdots[dotstyle=*](11.,-3.)
\psdots[dotstyle=*](10.,-4.)
\psdots[dotstyle=*](12.,-3.)
\psdots[dotstyle=*](11.,-4.)
\psdots[dotstyle=*](10.,-5.)
\psdots[dotstyle=*](9.,-3.)
\psdots[dotstyle=*](11.,-5.)
\psdots[dotstyle=*](10.,-3.)
\psdots[dotstyle=*](12.,-5.)
\psdots[dotstyle=*](11.,-3.)
\psdots[dotstyle=*](9.,-5.)
\psdots[dotstyle=*](10.,-3.)
\psdots[dotstyle=*](10.,-5.)
\psdots[dotstyle=*](11.,-3.)
\psdots[dotstyle=*](11.,-5.)
\psdots[dotstyle=*](12.,-3.)
\psdots[dotstyle=*](9.,-4.)
\psdots[dotstyle=*](10.,-4.)
\psdots[dotstyle=*](10.,-4.)
\psdots[dotstyle=*](11.,-4.)
\psdots[dotstyle=*](11.,-4.)
\psdots[dotstyle=*](12.,-4.)
\psdots[dotstyle=*](9.,-5.)
\psdots[dotstyle=*](10.,-5.)
\psdots[dotstyle=*](10.,-5.)
\psdots[dotstyle=*](11.,-5.)
\psdots[dotstyle=*](11.,-5.)
\psdots[dotstyle=*](12.,-5.)
\psdots[dotstyle=*](9.,-5.)
\psdots[dotstyle=*](10.,-3.)
\psdots[dotstyle=*](9.,-4.)
\psdots[dotstyle=*](10.,-5.)
\psdots[dotstyle=*](10.,-4.)
\psdots[dotstyle=*](11.,-5.)
\psdots[dotstyle=*](11.,-4.)
\psdots[dotstyle=*](12.,-5.)
\psdots[dotstyle=*](10.,-4.)
\psdots[dotstyle=*](9.,-5.)
\psdots[dotstyle=*](11.,-4.)
\psdots[dotstyle=*](10.,-5.)
\psdots[dotstyle=*](12.,-4.)
\psdots[dotstyle=*](11.,-5.)
\psdots[dotstyle=*](12.,-3.)
\psdots[dotstyle=*](12.,-3.)
\psdots[dotstyle=*](12.,-3.)
\psdots[dotstyle=*](12.,-4.)
\psdots[dotstyle=*](12.,-4.)
\psdots[dotstyle=*](12.,-5.)
\psdots[dotstyle=*](12.,-5.)
\psdots[dotstyle=*](12.,-5.)
\end{scriptsize}
\end{pspicture*}
  \vspace{-.5cm}  \caption{The graph $C_9[3K_1]$ and the base face of the reflexible non-orientable map constructed in Proposition~\ref{pro:nonor_s_odd}.}
    \label{fig:non-oreintable-sodd-face}
\end{figure}
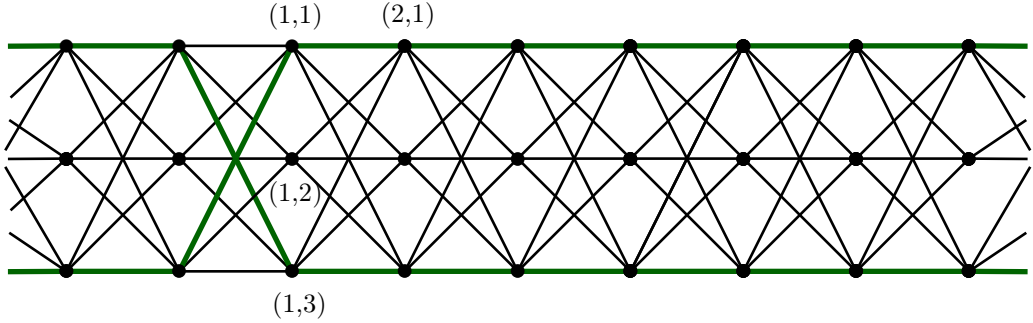

Moreover, $v \sigma_2= (1z,1t)=(1,1) = v$, while for each $j$ and $k$ with $1\leq j,k\leq m$ we see that
 $(2,j)\sigma_2 = (n,j-1)$ and $(n,k)\sigma_2 = (2,k-3)$, implying that $\sigma_2$ stabilises $v$ and cyclically permutes the vertices $(2,1), (n,m), (2,m-3), (n,m-4), \dots, (n,4)$.
  Since $m$ is odd,  this implies that $\langle \sigma_2\rangle$ acts transitively on the neighbourhood of $v$.
We can thus use Theorem~\ref{prop:undergraph} to deduce that $G$ acts transitively on the arcs of $\Gamma$.

Moreover, by Lemma~\ref{lemma:nonor_s_odd}, the order of $G$ is twice the number of arcs in $\G$, implying that $|\mathrm{Stab}_{G}(uw)| = 2$.
Let us now find an involution $\rho\in G$ fixing $v$ and inverting $\sigma_1$; the result will then follow by  Theorem~\ref{prop:undergraph}.

Note that $c^{(m-1)(m+1)} = c^{-1}$ and $c^{(m+1)^2}=c$. Therefore, defining $\varphi, \psi$ as in the proof of Lemma~\ref{lemma:nonor_s_odd},  one can verify that~\eqref{eq:nonor_varphi} and~\eqref{eq:nonor_psi} imply that
$$ 
\varphi^{(m-1)(m+1)/4}\psi^{(m+1)^2/4} = (c,c,c^{-3},c^5,\ldots , c^{(-1)^{i-1}(3-2i)}, \ldots , c^3)1.
$$
Thus~\eqref{eq:nonor_s1} implies that setting $\rho = \sigma_1^n\varphi^{(m-1)(m+1)/4}\psi^{(m+1)^2/4}\sigma_2$ we have $\rho \in G$ and 
\begin{equation}
\label{eq:reflex_tau_s_odd} 
	\rho = (1,tc^{-1},tc^{-1},\ldots , tc^{-1})z.
\end{equation}
Observe that $(1,1)\rho = (1,1)$ and that $\rho$ is an involution. Since
\begin{equation}
\label{eq:tau_sigma_1}
	\rho \sigma_1 = (1,tc^{-1},tc^{-1},\ldots , tc^{-1})(1,tc^{-1},1,\ldots , 1)zr = (1,1,tc^{-1},\ldots,tc^{-1})zr,
\end{equation}
it follows that $\rho\sigma_1$ is also an involution, showing that $\rho$ inverts $\sigma_1$, as desired.
\end{proof}

Let us now use the Petrie dual construction described in Lemma~\ref{lem:Petrie} to deduce existence of
an orientable reflexible map with the skeleton $\G$.
Let $\sigma_1$, $\sigma_2$ be as in~\eqref{eq:ref_nor_s_odd}, and let $\rho$ be as in~\eqref{eq:reflex_tau_s_odd}. Let 
\begin{equation}
\label{eq:ref_ori_s_odd}
    \eta_1 = \sigma_1\sigma_2\rho, \hspace{1cm} \eta_2 = \sigma_2 \hspace{1cm} \text{and}  \hspace{1cm} H = \langle \eta_1, \eta_2\rangle.
\end{equation}

\begin{proposition}
\label{pro:reflex_or_s_odd}
 Let $s$ and $m$ be odd integers with $m \geq 3$, let $n = sm$ and
  let $\eta_1$, $\eta_2$ and $H$ be as in~\eqref{eq:ref_ori_s_odd}.
  Then $\G = C_n[mK_1]$ is the skeleton of a polytopal orientable reflexible map $\m$ of type $\{n, 2m\}$ with $\Aut^+(\m) = H$.
\end{proposition}

\begin{proof}
By Lemma~\ref{lem:Petrie} and Proposition~\ref{pro:nonor_s_odd}, we already know that $\G$ is the skeleton of a reflexible map $\m^\pi$ with $\Aut^+(\m^\pi) = H$ with the distinguished generators of a base flag
being $\eta_1$ and $\eta_2$. 
To conclude the proof we only need to prove that $|\eta_1| = n$ and that $\m^\pi$ is orientable. 

Let $\sigma_1$ and $\sigma_2$ be as in~\eqref{eq:ref_nor_s_odd}, and let $\rho$ be as in~\eqref{eq:reflex_tau_s_odd}.
From~\eqref{eq:reflex_s1s2_s_odd},~\eqref{eq:reflex_tau_s_odd} and~\eqref{eq:ref_ori_s_odd} we easily see that
$$
\eta_1 = (t, tc^{-4}, tc^4, \ldots, tc^{(-1)^{i+1}(2i-1)-1}, \ldots, tc^2, c)r.
$$
Observe that multiplying the components of $(t, tc^{-4}, tc^4, \ldots, tc^{(-1)^{i+1}(2i-1)-1}, \ldots, tc^2, c)$ in this order, we obtain
$$
	c^{-4}c^{-12}c^{-20}\cdots c^{-4(n-2)}c = c^\ell,
$$ 
where
$$
	\ell = -4(1+3+5+\cdots + (n-2)) + 1 = -4((n-1)/2)^2 + 1 = n(2-n).
$$
Therefore, $c^\ell = 1$, and it now easily follows that $\eta_1^n = 1$, proving that $|\eta_1| = n$, as claimed.

Let us now show that $\m^\pi$ is orientable. Observe that it suffices to show that $H$ is of order $2m^2n$.
Let $\zeta = \eta_2^2$ and $\xi = \eta_1^{-1}\zeta\eta_1$. 
Since $\sigma_1\sigma_2 = \sigma_2^{-1}\sigma_1^{-1}$ and since $\rho$ inverts $\sigma_1$ and $\sigma_2$, we find that 
$$
	\xi = \rho\sigma_2^{-1}\sigma_1^{-1}\sigma_2^2\sigma_1\sigma_2\rho = \rho\sigma_1\sigma_2^2\sigma_1^{-1}\rho = \sigma_1^{-1}\sigma_2^{-2}\sigma_1. 
$$
Therefore, $\langle \zeta, \xi \rangle = H_1 \cong C_m \times C_m$, where $H_1$ is as in the proof of Lemma~\ref{lemma:nonor_s_odd}.
Setting $K_2 = \langle \zeta, \xi, \eta_1 \rangle$ we thus see that $K_2$ is a group of order $m^2 n$ (recall that $H_1$ is normal in $\langle \sigma_1, \sigma_2 \rangle$). Moreover, as in the proof of Proposition~\ref{pro:nonor_s_odd}, the fact that this time $\eta_1\eta_2$ is an involution implies that
$$
	\eta_2^{-1}\eta_1\eta_2 = \eta_2^{-1}\eta_2^{-1}\eta_1^{-1} = \zeta^{-1}\eta_1^{-1} \in K_2,
$$
showing that $K_2$ is normal in $H$. Since $\eta_2^2 \in K_2$, this implies that $K_2$ is of index $2$ in $H$, and consequently $|H| = 2m^2n$, as claimed.
\end{proof}



\begin{figure}
    \centering
    \newrgbcolor{ududff}{0.30196078431372547 0.30196078431372547 1.}
\newrgbcolor{qqwuqq}{0. 0.39215686274509803 0.}
\psset{xunit=1.5cm,yunit=1.5cm,algebraic=true,dimen=middle,dotstyle=o,dotsize=5pt 0,linewidth=1.6pt,arrowsize=3pt 2,arrowinset=0.25}
\begin{pspicture*}(2.583761845572507,-5.559524410791097)(13.470226437165865,-2.3080360925054935)
\psline[linewidth=1.pt](4.,-3.)(5.,-3.)
\psline[linewidth=1.pt](5.,-3.)(6.,-3.)
\psline[linewidth=2.pt,linecolor=qqwuqq](6.,-3.)(7.,-3.)
\psline[linewidth=1.pt](7.,-3.)(8.,-3.)
\psline[linewidth=1.pt](8.,-3.)(9.,-3.)
\psline[linewidth=2.pt,linecolor=qqwuqq](4.,-3.)(3.4813532062736874,-3.0030743857047866)
\psline[linewidth=1.pt](4.,-4.)(5.,-4.)
\psline[linewidth=1.pt](4.,-4.)(3.4813532062736874,-4.003074385704787)
\psline[linewidth=1.pt](5.,-4.)(6.,-4.)
\psline[linewidth=1.pt](6.,-4.)(7.,-4.)
\psline[linewidth=1.pt](7.,-4.)(8.,-4.)
\psline[linewidth=1.pt](8.,-4.)(9.,-4.)
\psline[linewidth=1.pt](4.,-5.)(3.4813532062736874,-5.003074385704787)
\psline[linewidth=1.pt](4.,-5.)(5.,-5.)
\psline[linewidth=1.pt](5.,-5.)(6.,-5.)
\psline[linewidth=1.pt](6.,-5.)(7.,-5.)
\psline[linewidth=1.pt](7.,-5.)(8.,-5.)
\psline[linewidth=2.pt,linecolor=qqwuqq](8.,-5.)(9.,-5.)
\psline[linewidth=1.pt](4.,-3.)(3.5047403023605423,-3.4591227593984533)
\psline[linewidth=1.pt](4.,-4.)(3.4930467543171146,-3.657913076136718)
\psline[linewidth=1.pt](4.,-3.)(3.457966110186833,-3.9268646811355468)
\psline[linewidth=1.pt](4.,-4.)(3.5047403023605423,-4.459122759398453)
\psline[linewidth=1.pt](4.,-5.)(3.4930467543171146,-4.657913076136718)
\psline[linewidth=1.pt](4.,-3.)(5.,-4.)
\psline[linewidth=1.pt](5.,-4.)(6.,-3.)
\psline[linewidth=1.pt](6.,-3.)(7.,-4.)
\psline[linewidth=1.pt](7.,-4.)(8.,-3.)
\psline[linewidth=1.pt](8.,-3.)(9.,-4.)
\psline[linewidth=1.pt](9.,-4.)(8.,-5.)
\psline[linewidth=1.pt](8.,-5.)(7.,-4.)
\psline[linewidth=1.pt](7.,-4.)(6.,-5.)
\psline[linewidth=1.pt](6.,-5.)(5.,-4.)
\psline[linewidth=1.pt](5.,-4.)(4.,-5.)
\psline[linewidth=1.pt](4.,-5.)(5.,-3.)
\psline[linewidth=1.pt](5.,-3.)(6.,-5.)
\psline[linewidth=1.pt](6.,-5.)(7.,-3.)
\psline[linewidth=2.pt,linecolor=qqwuqq](7.,-3.)(8.,-5.)
\psline[linewidth=1.pt](8.,-5.)(9.,-3.)
\psline[linewidth=1.pt](9.,-3.)(8.,-4.)
\psline[linewidth=1.pt](8.,-4.)(7.,-3.)
\psline[linewidth=1.pt](7.,-3.)(6.,-4.)
\psline[linewidth=1.pt](6.,-4.)(5.,-3.)
\psline[linewidth=1.pt](5.,-3.)(4.,-4.)
\psline[linewidth=1.pt](4.,-4.)(5.,-5.)
\psline[linewidth=1.pt](5.,-5.)(6.,-4.)
\psline[linewidth=1.pt](6.,-4.)(7.,-5.)
\psline[linewidth=1.pt](7.,-5.)(8.,-4.)
\psline[linewidth=1.pt](8.,-4.)(9.,-5.)
\psline[linewidth=1.pt](9.,-5.)(8.,-3.)
\psline[linewidth=1.pt](8.,-3.)(7.,-5.)
\psline[linewidth=1.pt](7.,-5.)(6.,-3.)
\psline[linewidth=2.pt,linecolor=qqwuqq](6.,-3.)(5.,-5.)
\psline[linewidth=2.pt,linecolor=qqwuqq](5.,-5.)(4.,-3.)
\rput[tl](5.787079966550198,-2.62114237500707){(1,1)}
\rput[tl](5.787079966550198,-4.198716336841937){(1,2)}
\rput[tl](5.8232076145311495,-5.186205381654601){(1,3)}
\rput[tl](6.786611560689854,-2.6090998256800866){(2,1)}
\psline[linewidth=1.pt](4.,-5.)(3.457966110186833,-4.073135318864454)
\psline[linewidth=1.pt](9.,-3.)(10.,-3.)
\psline[linewidth=1.pt](10.,-3.)(11.,-3.)
\psline[linewidth=2.pt,linecolor=qqwuqq](11.,-3.)(12.,-3.)
\psline[linewidth=1.pt](9.,-3.)(10.,-4.)
\psline[linewidth=1.pt](10.,-3.)(11.,-4.)
\psline[linewidth=1.pt](11.,-3.)(12.,-4.)
\psline[linewidth=1.pt](10.,-3.)(9.,-4.)
\psline[linewidth=1.pt](11.,-3.)(10.,-4.)
\psline[linewidth=1.pt](12.,-3.)(11.,-4.)
\psline[linewidth=1.pt](10.,-5.)(9.,-3.)
\psline[linewidth=1.pt](11.,-5.)(10.,-3.)
\psline[linewidth=1.pt](12.,-5.)(11.,-3.)
\psline[linewidth=1.pt](9.,-5.)(10.,-3.)
\psline[linewidth=2.pt,linecolor=qqwuqq](10.,-5.)(11.,-3.)
\psline[linewidth=1.pt](11.,-5.)(12.,-3.)
\psline[linewidth=1.pt](9.,-4.)(10.,-4.)
\psline[linewidth=1.pt](10.,-4.)(11.,-4.)
\psline[linewidth=1.pt](11.,-4.)(12.,-4.)
\psline[linewidth=2.pt,linecolor=qqwuqq](9.,-5.)(10.,-5.)
\psline[linewidth=1.pt](10.,-5.)(11.,-5.)
\psline[linewidth=1.pt](11.,-5.)(12.,-5.)
\psline[linewidth=1.pt](9.,-5.)(10.,-3.)
\psline[linewidth=1.pt](9.,-4.)(10.,-5.)
\psline[linewidth=1.pt](10.,-4.)(11.,-5.)
\psline[linewidth=1.pt](11.,-4.)(12.,-5.)
\psline[linewidth=1.pt](10.,-4.)(9.,-5.)
\psline[linewidth=1.pt](11.,-4.)(10.,-5.)
\psline[linewidth=1.pt](12.,-4.)(11.,-5.)
\psline[linewidth=2.pt,linecolor=qqwuqq](12.,-3.)(12.518646793726312,-3.003074385704786)
\psline[linewidth=1.pt](12.,-3.)(12.495259697639458,-3.459122759398453)
\psline[linewidth=1.pt](12.,-3.)(12.542033889813169,-3.9268646811355463)
\psline[linewidth=1.pt](12.,-4.)(12.506953245682887,-3.6579130761367176)
\psline[linewidth=1.pt](12.,-4.)(12.518646793726314,-4.003074385704786)
\psline[linewidth=1.pt](12.,-5.)(12.542033889813169,-4.073135318864453)
\psline[linewidth=1.pt](12.,-5.)(12.506953245682887,-4.657913076136717)
\psline[linewidth=1.pt](12.,-5.)(12.518646793726314,-5.003074385704786)
\begin{scriptsize}
\psdots[dotstyle=*](4.,-3.)
\psdots[dotstyle=*](5.,-3.)
\psdots[dotstyle=*](6.,-3.)
\psdots[dotstyle=*](7.,-3.)
\psdots[dotstyle=*](8.,-3.)
\psdots[dotstyle=*](9.,-3.)
\psdots[dotstyle=*](4.,-4.)
\psdots[dotstyle=*](4.,-4.)
\psdots[dotstyle=*](5.,-4.)
\psdots[dotstyle=*](4.,-4.)
\psdots[dotstyle=*](5.,-4.)
\psdots[dotstyle=*](6.,-4.)
\psdots[dotstyle=*](6.,-4.)
\psdots[dotstyle=*](7.,-4.)
\psdots[dotstyle=*](7.,-3.)
\psdots[dotstyle=*](8.,-3.)
\psdots[dotstyle=*](7.,-3.)
\psdots[dotstyle=*](8.,-3.)
\psdots[dotstyle=*](7.,-4.)
\psdots[dotstyle=*](8.,-4.)
\psdots[dotstyle=*](8.,-4.)
\psdots[dotstyle=*](9.,-4.)
\psdots[dotstyle=*](9.,-4.)
\psdots[dotstyle=*](4.,-5.)
\psdots[dotstyle=*](4.,-5.)
\psdots[dotstyle=*](5.,-5.)
\psdots[dotstyle=*](5.,-5.)
\psdots[dotstyle=*](6.,-5.)
\psdots[dotstyle=*](6.,-5.)
\psdots[dotstyle=*](7.,-5.)
\psdots[dotstyle=*](7.,-5.)
\psdots[dotstyle=*](8.,-5.)
\psdots[dotstyle=*](8.,-5.)
\psdots[dotstyle=*](9.,-5.)
\psdots[dotstyle=*](9.,-5.)
\psdots[dotstyle=*](4.,-4.)
\psdots[dotstyle=*](4.,-4.)
\psdots[dotstyle=*](4.,-4.)
\psdots[dotstyle=*](4.,-4.)
\psdots[dotstyle=*](4.,-5.)
\psdots[dotstyle=*](9.,-3.)
\psdots[dotstyle=*](9.,-3.)
\psdots[dotstyle=*](9.,-4.)
\psdots[dotstyle=*](9.,-4.)
\psdots[dotstyle=*](9.,-5.)
\psdots[dotstyle=*](9.,-4.)
\psdots[dotstyle=*](4.,-5.)
\psdots[dotstyle=*](10.,-3.)
\psdots[dotstyle=*](10.,-4.)
\psdots[dotstyle=*](10.,-5.)
\psdots[dotstyle=*](11.,-3.)
\psdots[dotstyle=*](11.,-4.)
\psdots[dotstyle=*](11.,-5.)
\psdots[dotstyle=*](12.,-3.)
\psdots[dotstyle=*](12.,-4.)
\psdots[dotstyle=*](12.,-5.)
\psdots[dotstyle=*](9.,-3.)
\psdots[dotstyle=*](10.,-3.)
\psdots[dotstyle=*](10.,-3.)
\psdots[dotstyle=*](11.,-3.)
\psdots[dotstyle=*](11.,-3.)
\psdots[dotstyle=*](12.,-3.)
\psdots[dotstyle=*](9.,-3.)
\psdots[dotstyle=*](10.,-4.)
\psdots[dotstyle=*](10.,-3.)
\psdots[dotstyle=*](11.,-4.)
\psdots[dotstyle=*](11.,-3.)
\psdots[dotstyle=*](12.,-4.)
\psdots[dotstyle=*](10.,-3.)
\psdots[dotstyle=*](9.,-4.)
\psdots[dotstyle=*](11.,-3.)
\psdots[dotstyle=*](10.,-4.)
\psdots[dotstyle=*](12.,-3.)
\psdots[dotstyle=*](11.,-4.)
\psdots[dotstyle=*](10.,-5.)
\psdots[dotstyle=*](9.,-3.)
\psdots[dotstyle=*](11.,-5.)
\psdots[dotstyle=*](10.,-3.)
\psdots[dotstyle=*](12.,-5.)
\psdots[dotstyle=*](11.,-3.)
\psdots[dotstyle=*](9.,-5.)
\psdots[dotstyle=*](10.,-3.)
\psdots[dotstyle=*](10.,-5.)
\psdots[dotstyle=*](11.,-3.)
\psdots[dotstyle=*](11.,-5.)
\psdots[dotstyle=*](12.,-3.)
\psdots[dotstyle=*](9.,-4.)
\psdots[dotstyle=*](10.,-4.)
\psdots[dotstyle=*](10.,-4.)
\psdots[dotstyle=*](11.,-4.)
\psdots[dotstyle=*](11.,-4.)
\psdots[dotstyle=*](12.,-4.)
\psdots[dotstyle=*](9.,-5.)
\psdots[dotstyle=*](10.,-5.)
\psdots[dotstyle=*](10.,-5.)
\psdots[dotstyle=*](11.,-5.)
\psdots[dotstyle=*](11.,-5.)
\psdots[dotstyle=*](12.,-5.)
\psdots[dotstyle=*](9.,-5.)
\psdots[dotstyle=*](10.,-3.)
\psdots[dotstyle=*](9.,-4.)
\psdots[dotstyle=*](10.,-5.)
\psdots[dotstyle=*](10.,-4.)
\psdots[dotstyle=*](11.,-5.)
\psdots[dotstyle=*](11.,-4.)
\psdots[dotstyle=*](12.,-5.)
\psdots[dotstyle=*](10.,-4.)
\psdots[dotstyle=*](9.,-5.)
\psdots[dotstyle=*](11.,-4.)
\psdots[dotstyle=*](10.,-5.)
\psdots[dotstyle=*](12.,-4.)
\psdots[dotstyle=*](11.,-5.)
\psdots[dotstyle=*](12.,-3.)
\psdots[dotstyle=*](12.,-3.)
\psdots[dotstyle=*](12.,-3.)
\psdots[dotstyle=*](12.,-4.)
\psdots[dotstyle=*](12.,-4.)
\psdots[dotstyle=*](12.,-5.)
\psdots[dotstyle=*](12.,-5.)
\psdots[dotstyle=*](12.,-5.)
\end{scriptsize}
\end{pspicture*}{}
    \caption{The graph $C_9[3K_1]$ and the base face of the reflexible orientable map constructed in Proposition~\ref{pro:reflex_or_s_odd}}
    \label{fig:oreintable-sodd-face}
\end{figure}
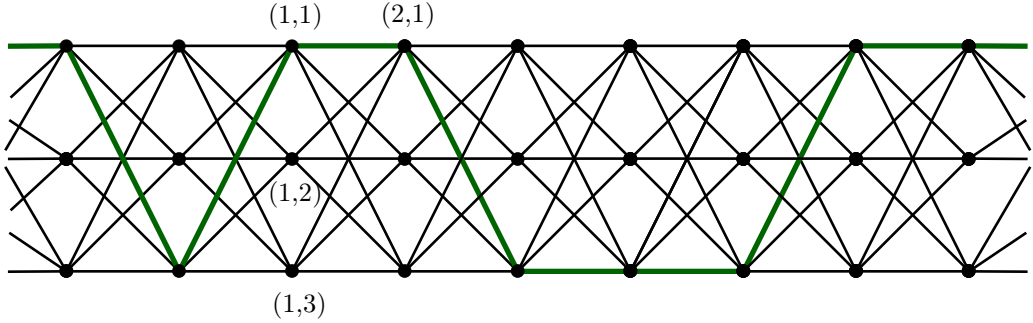

\subsection{The examples with $s$ even}

Throughout this subsection let $s$ be an even number not divisible by $4$. It turns out that this time the orientable and non-orientable reflexible maps that we  construct are not the Petrie duals of one another (in fact, the non-orientable one is self-Petrie). 
Thus, we construct each of them in a rather different way.

We start with the construction for the non-orientable maps by defining
\begin{eqnarray}
\label{eq:ref_nor_s_even}
	\sigma_1 &=& (1,1,\ldots , 1, tc^{-1})r,\nonumber \\ 
 \sigma_2 &=& (t,c^{-1},c,c,c^{-1},c^{-1},\ldots , c, c, c^{-1}, c^{-1})z, \ \text{and} \\
 G&=&\langle \sigma_1, \sigma_2 \rangle,
 \nonumber
\end{eqnarray}
where for $3 \leq i \leq n$ the $i$-th component of the ``$n$-tuple'' for $\sigma_2$ is $c$ or $c^{-1}$, depending on whether $i$ is congruent to one of $0$ and $3$ modulo $4$, or one of $1$ and $2$ modulo $4$, respectively. Note that the $\sigma_1$ from~\eqref{eq:ref_nor_s_even} has the exact same form as the $\sigma_1$ defined in~\eqref{eq:ref_nor_s_odd}.

\begin{lemma}
    \label{lemma:reflex_nonor_s_even}
Let 
$\sigma_1$, $\sigma_2$ and $G$ be as in~\eqref{eq:ref_nor_s_even}.
Then the following holds:
\begin{equation*}
|\sigma_1| = 2n, \hspace{1cm} |\sigma_2| = 2m,\hspace{1cm} |\sigma_1\sigma_2| = 2 \hspace{1cm}\text{and} \hspace{1cm} |G| = 4m^2 n.
\end{equation*}
\end{lemma}

\begin{proof}
As in the proof of Lemma~\ref{lemma:nonor_s_odd}, we can deduce that $|\sigma_1| = 2n$. Set $\varphi = \sigma_2^2$. Since $n$ is twice an odd number it is easy to see that 
\begin{equation}
\label{eq:varphi_nonor_s_even}
\varphi = (1,c^{-2},1,c^2,1,\ldots , c^2, 1, c^{-2})1,
\end{equation}
and so the fact that $m$ is odd implies that $|\sigma_2| = 2m$, as claimed. We also see that
$$
	\sigma_1\sigma_2 = (1,1,\ldots , 1, tc^{-1})(c^{-1},c,c,\ldots , c^{-1}, c^{-1},t)rz = (c^{-1},c,c,\ldots , c^{-1}, c^{-1},c)rz,
$$
where the $i$-th component of the $n$-tuple is $c^{-1}$ whenever $i$ is congruent to $0$ or $1$ modulo $4$, and is $c$ otherwise. It is now easy to verify that $\sigma_1\sigma_2$ is an involution. As in the proof of Lemma~\ref{lemma:nonor_s_odd} set $\psi = \sigma_1^{-1} \varphi \sigma_1$. A straightforward calculation shows that
$$
	\psi = (c^2,1,c^{-2},1,\ldots , c^{-2},1,c^2,1)1,
$$
where the $i$-th components of the $n$-tuple are all equal to $1$ for $i$ even, while for $i$ odd they alternate between $c^2$ and $c^{-2}$. In view of~\eqref{eq:varphi_nonor_s_even} it now follows that $H_1 = \langle \varphi, \psi\rangle \cong C_m \times C_m$. Clearly, $\sigma_1^{-1}\varphi\sigma_1, \sigma_2^{-1}\varphi\sigma_2 \in H_1$. Moreover, 
$$
	\sigma_1^{-1}\psi \sigma_1 = (tc^{-1},1,\ldots , 1)(1,c^2,1,c^{-2},\ldots , 1, c^2)(tc^{-1},1,\ldots , 1)1,
$$
and so in fact $\sigma_1^{-1}\psi \sigma_1 = \varphi^{-1}$. Repeating the computation from~\eqref{eq:sig2psi} we also see that $\sigma_2^{-1}\psi \sigma_2 \in H_1$, showing that $H_1$ is normal in $G$. It is clear that $H_1 \cap \langle \sigma_1 \rangle = 1$, and so $K_1 = \langle \sigma_1, \varphi, \psi\rangle$ is a group of order $2m^2 n$. Just like in~\eqref{eq:sig2sig1_conj} we see that $\sigma_2^{-1}\sigma_1\sigma_2 \in K_1$, proving that $K_1$ is normal in $G$. Finally, since $\sigma_2 \notin K_1$ and $\sigma_2^2 \in K_1$, we see that $G$ is of order $4m^2n$, as claimed. 
\end{proof}

\begin{proposition}
\label{pro:reflex_nonor_s_even}
     Let $m\geq 3$ be an odd integer and let $n = sm$, with $s$ an even integer not divisible by $4$.
     Let $\sigma_1$, $\sigma_2$ and $G$ be as in~\eqref{eq:ref_nor_s_even}. Then $\G = C_n[mK_1]$ is the skeleton of a polytopal non-orientable reflexible map $\m$ of type $\{2n, 2m\}$ with $\Aut(\m) = \Aut^+(\m) = G$. 
\end{proposition}

\begin{proof}
In the same way as in the proof of Proposition~\ref{pro:nonor_s_odd}, we observe that $\sigma_1$ maps the vertex $v:=(1,1)$ to its neighbour $w:=(2,1)$ and that the orbit of $v$ under $\langle \sigma_1\rangle$ constitutes the cycle $f$ of length $2n$ with vertices $(1,1), (2,1), \dots, (n,1), (1,m), (2,m), \dots , (n,m)$, in that order.   
Further, note that  $v\sigma_2 = v$, and that for
every $j$ and $k$ with $1 \leq j,k \leq m$ the action of $\sigma_2$ on the neighbours $(2,j)$ and $(n,k)$ of $v$ is given by
 $(2,j)\sigma_2 = (n,j-1)$ and  $(n,k)\sigma_2 = (2,k-1)$. Since $m$ is odd, this shows that  
$\sigma_2$ cyclically permutes the neighbours $(2,1), (n,m), (2,m-1), (n,m-2), \dots, (n,2)$ of $v$. By Theorem~\ref{prop:undergraph} and Lemma~\ref{lemma:reflex_nonor_s_even} the group $G$ acts transitively on the arcs of $\Gamma$ and has twice as many elements as there are arcs of $\G$, implying that $|\mathrm{Stab}_{G}(vw)| = 2$.

To complete the proof set $\rho = \sigma_1^n\varphi^{(m-1)/2}\psi^{(m+1)/2}\sigma_2$, where $\varphi$ and $\psi$ are as in the proof of Lemma~\ref{lemma:reflex_nonor_s_even},
    and note that $\rho \in G$. It is easy to see that
$$
	\sigma_1^n\varphi^{(m-1)/2}\psi^{(m+1)/2} = (t,t,tc^{-2},tc^{-2},t,t,\ldots , tc^{-2},tc^{-2},t,t)1,
$$
from which one easily verifies that $\rho$ in fact has the same form as in~\eqref{eq:reflex_tau_s_odd}, that is
$$
	\rho = (1,tc^{-1},tc^{-1},\ldots , tc^{-1})z,
$$
which is an involution. That $\rho\sigma_1$ is also an involution can be verified just as in~\eqref{eq:tau_sigma_1}. Therefore, $\rho$ inverts $\sigma_1$.
Moreover, it is clear that $\rho$ fixes $v$. The result now follows by part (b) of Theorem~\ref{prop:undergraph}.
\end{proof}

As already remarked it turns out that the Petrie dual of the map corresponding to $\sigma_1$ and $\sigma_2$ from Proposition~\ref{pro:reflex_nonor_s_even} results in the same map. We thus need a different construction for a reflexible map on an orientable surface. To this end we set
\begin{eqnarray}
\label{eq:reflex_eta1_eta2_s_even}
	\eta_1 & = & (t,t,\ldots , t)r, \nonumber\\
	\eta_2 & = & (t,t,tc,tc,tc^2,tc^2,\ldots , tc^{-1}, tc^{-1})z, \ \text{and} \\
	H & = & \langle \eta_1, \eta_2 \rangle, \nonumber
\end{eqnarray}
where for $\eta_2$ the $(2i+1)$-th and $(2i+2)$-th component of the ``$n$-tuple'' is $tc^i$ for each $i$ with $0 \leq i < n/2$. Note that $n$ is even (since $s$ is), and thus for $i = \frac{n}{2}-1$ we see that $c^i = c^{\frac{n}{2}-1} = c^{-1}$, since $m$ divides $\frac{n}{2}$. Therefore, the last two components are indeed $tc^{-1}$.

\begin{lemma}
    \label{lemma:reflex_or_s_even}
Let $\eta_1, \eta_2$ and $H$ be as in~\eqref{eq:reflex_eta1_eta2_s_even}. Then the following holds:
\begin{equation*}
|\eta_1| = n, \hspace{1cm} |\eta_2| = 2m,\hspace{1cm} |\eta_1\eta_2| = 2 \hspace{1cm}\text{and} \hspace{1cm} |H| = 2m^2 n.
\end{equation*}
\end{lemma}

\begin{proof}
That $|\eta_1| = n$ is clear (recall that $n$ is even). Next,
$$
	\eta_2^2 = (t,t,tc,tc,tc^2,tc^2,\ldots , tc^{-1},tc^{-1})(t,tc^{-1},tc^{-1},tc^{-2},tc^{-2},\ldots , tc,tc,t)1,
$$
and so setting $\zeta = \eta_2^2$ we deduce
\begin{equation}
\label{eq:zeta_s_even}
	\zeta = (1,c^{-1},c^{-2},\ldots, c^{1-i}, \ldots, c)1,
\end{equation}
which clearly shows that $|\eta_2| = 2m$. It is also easy to see that
$$
	\eta_1\eta_2 = (1,c,c,c^2,c^2,\ldots , c^{-1},c^{-1},1)rz,
$$
which is easily seen to be an involution.

Let $\xi = \eta_1^{-1}\zeta\eta_1$ and note that
\begin{equation}
\label{eq:xi_s_even}
	\xi = (t,t,\ldots , t)(c,1,c^{-1},\ldots , c^{2-i}, \ldots , c^2)(t,t,\ldots , t)1 = (c^{-1},1,c^{1},\ldots , c^{i-2}, \ldots , c^{-2})1.
\end{equation}
It follows that $H_2 = \langle \zeta, \xi\rangle \cong C_m \times C_m$. Since 
$$
	\eta_1^{-1}\xi\eta_1 = (t,t,\ldots , t)(c^{-2},c^{-1},1,\ldots , c^{i-3}, \ldots , c^{-3})(t,t,\ldots , t)1 = (c^{2},c^{1},1,\ldots , c^{3-i}, \ldots , c^{3})1,
$$
we see that $\eta_1^{-1}\xi\eta_1 = \zeta^{-1}\xi^{-2} \in H_2$, showing that $\eta_1$ normalizes $H_2$. Analogously to~\eqref{eq:sig2psi} we can verify that $\eta_2^{-1}\xi\eta_2 = \eta_1\zeta\eta_1^{-1} \in H_2$, showing that $H_2$ is normal in $H$. Since $H_2 \cap \langle \eta_1\rangle = 1$, we thus find that $K_2 = \langle \eta_1, \zeta, \xi\rangle$ is of order $m^2n$ and just as in~\eqref{eq:sig2sig1_conj} we see that $\eta_2^{-1}\eta_1\eta_2 = \zeta^{-1}\eta_1^{-1} \in K_2$, showing that $K_2$ is also normal in $H$. Finally, since clearly $\eta_2 \notin K_2$ but $\eta_2^2 \in K_2$, $H$ is a group of order $2m^2 n$, as claimed.
\end{proof}

\begin{proposition}
\label{pro:reflex_or_s_even}
Let $m \geq 3$ be an odd integer, let $s \geq 2$ be even but not divisible by $4$, and let $n = sm$. Furthermore, let $\eta_1, \eta_2$ and $H$ be as in~\eqref{eq:reflex_eta1_eta2_s_even}. Then $\G = C_n[mK_1]$ is the skeleton of a polytopal orientable reflexible map $\m$ of type $\{n, 2m\}$ with $\Aut^+(\m)=H$. 
\end{proposition}

\begin{proof}
Note that $\eta_1$ maps $v:=(1,1)$ to its neighbour $w:=(2,1)$, while $\eta_2$ permutes the neighbours of $v$ cyclically (in the order $(2,1), (n,1), (2,m), (n,2),(2,m-1),(n,3), \ldots , (n,m)$). Theorem~\ref{prop:undergraph} and Lemma~\ref{lemma:reflex_or_s_even} thus imply that $H$ acts regularly on the arcs of $\Gamma$. Consequently, part (a) of Theorem~\ref{prop:undergraph} in fact implies that the graph $\G$ is the skeleton of a polytopal rotary map $\m$ of type $\{n, 2m\}$ on an orientable surface with $\Aut^+(\m) = H$.

To complete the proof, set $\rho = z$ and observe that it is an involution fixing $v$. Since the automorphisms $\eta_1\rho$ and $\eta_2\rho$ are clearly both involutions, $\rho$ inverts each of $\eta_1$ and $\eta_2$ and therefore normalizes $H$. Part (a) of Theorem~\ref{prop:undergraph} thus implies that $\m$ is reflexible.  
%
\end{proof}

\section{Proof of the main theorem and concluding remarks}
\label{sec:concluding}

The proof of Theorem~\ref{the:main}  follows directly from Propositions~\ref{pro:rotary}, \ref{pro:nonor_s_odd}, \ref{pro:reflex_or_s_odd}, \ref{pro:reflex_nonor_s_even} and~\ref{pro:reflex_or_s_even}; the genus of the maps can of course be calculated from the order of the skeleton and the type of the map.

Let us conclude the paper by the following remarks.

\begin{enumerate}
\item
A recently computed census of rotary maps  \cite{mappaper} shows that there are precisely 282 simple graphs with at most 3000 edges that embed as skeletons of all three classes of rotary maps. The constructions provided in this paper cover about $60\%$ of them. A natural question that arises is whether one can classify all simple graphs of this type.
\item
Further, recall that a polytopal map is {\em polyhedral} provided that every two distinct faces intersect in a vertex, an edge or not at all (see \cite{egon}).
While all the maps constructed in this paper are polytopal, not all are polyhedral. For example, the faces of the chiral maps constructed in Section~\ref{sec:chiral},
are all hamiltonian cycles, implying that any two faces meet in all the vertices of their skeletons.
As the census \cite{mappaper} shows, there are only two graphs with at most 3000 edges that
embed as polyhedral rotary maps in all three possible ways. Both have 576 vertices, are of valence 6 and are non-bipartite.
They are denoted as PlhSk$(576,27)$ and PlhSk$(576,32)$ in the list of all graphs that embed as polyhedral rotary maps, provided at  \cite{mapcensus}.
This raises an obvious question whether one can find an infinite family of such graphs.
\end{enumerate}

\medskip\bigskip\noindent 
{\Large\bf Acknowledgements}

\medskip\smallskip
\noindent 

The first author acknowledges the support from PAPIIT-DGAPA, UNAM (grant IN-109023) and from CONACyT  (grant A1-S-21678). The second and third author acknowledge the support from the Slovenian Research and Innovation Agency (research programmes P1-0294 and P1-0285, respectively, and research projects J1-4351 and J1-50000, respectively).

\end{document}